\documentclass[journal]{IEEEtran}

\usepackage{graphicx}
\usepackage{amscd}
\usepackage{amssymb}
\usepackage{epsf}
\usepackage{hyperref}

\usepackage{amsmath,amsfonts}
\usepackage{mathrsfs} 
\usepackage{psfrag}
\usepackage{algorithm}  
\usepackage{algpseudocode}
\usepackage{subfigure}
\usepackage{pifont}
\usepackage{color}
\usepackage[amsmath,hyperref,thmmarks]{ntheorem}  
\usepackage{cite}
\usepackage{booktabs}

\usepackage{chngcntr} 

\newtheorem{definition}{Definition}[section]
\newtheorem{lemma}{Lemma}[section]
\newtheorem{theorem}{Theorem}[section]
\newtheorem{proposition}{Proposition}[section]
\newtheorem{corollary}{Corollary}[section]

\newcommand{\gen}[1]{^{(#1)}}
\renewcommand{\th}[1]{#1^\text{th}}

\counterwithout{paragraph}{subsubsection} 
\counterwithin{paragraph}{section} 

\begin{document}

\title{Reachability Analysis of Large Linear Systems with Uncertain Inputs in the Krylov Subspace}

\author{Matthias~Althoff
\thanks{Matthias Althoff is with the Department of Computer Science, 
        Technische Universit\"at M\"unchen, 85748 Garching, Germany,
        email: {\tt\small althoff@tum.de}}}

%

\maketitle

\begin{abstract}
 One often wishes for the ability to formally analyze large-scale systems---typically, however, one can either formally analyze a rather small system or informally analyze a large-scale system. This work tries to further close this performance gap for reachability analysis of linear systems. Reachability analysis can capture the whole set of possible solutions of a dynamic system and is thus used to prove that unsafe states are never reached; this requires full consideration of arbitrarily varying uncertain inputs, since sensor noise or disturbances usually do not follow any patterns. We use Krylov methods in this work to compute reachable sets for large-scale linear systems. While Krylov methods have been used before in reachability analysis, we overcome the previous limitation that inputs must be (piecewise) constant. As a result, we can compute reachable sets of systems with several thousand state variables for bounded, but arbitrarily varying inputs. 
\end{abstract}

\begin{IEEEkeywords}
Reachability analysis, linear systems, Krylov subspace, uncertain inputs, large-scale systems.
\end{IEEEkeywords}

\IEEEpeerreviewmaketitle

\section{Introduction}

Reachability analysis computes the set of possible solutions of dynamic systems subject to uncertain initial states and inputs. The availability of all possible solutions can be used for many purposes: formal verification of dynamic systems with discrete and/or continuous dynamics \cite{Banerjee2012, Asarin2006}; computation of invariance sets \cite{Kolmanovsky1998a, Blanchini1999}; computation of the region of attraction \cite{El-Guindy2017b}; optimization of constrained systems with uncertainties \cite{Limon2010, Schuermann2017a}; set-based observers \cite{Combastel2005, Le2013b}; and conformance checking \cite{Roehm2016}. The theory of efficiently computing reachable sets is advancing rapidly and their usefulness has already been demonstrated for many applications, such as automated driving \cite{Falcone2011, Althoff2014b}, robotics \cite{Taeubig2012, Pereira2016b}, power systems \cite{VillegasPico2014, El-Guindy2016}, and analog/mixed-signal circuits \cite{Frehse2006,Althoff2013c}. 

\paragraph{Reachability analysis of linear systems} The reachability analysis in particular of linear continuous systems has been intensely researched. While the main approach for computing reachable sets of linear systems by using the superposition principle has not significantly changed in recent years \cite[Eq.~4.14]{Dang2000}, much progress has been made by experimenting with different set representations: ellipsoids \cite{Kurzhanskiy2007}, polytopes \cite{Chutinan2003}, zonotopes \cite{Girard2005}, zonotope bundles \cite{Althoff2011f}, support functions \cite{Girard2008b}, level sets \cite{Mitchell2005}, and combination of support functions and zonotopes \cite{Althoff2016c}. In particular, when using zonotopes, support functions, or the combination thereof, one can efficiently compute systems with several hundred continuous state variables. Recently, a new technique has been proposed which combines simulation results by using the superposition principle to represent reachable states via generalized star sets for certain inputs \cite{Duggirala2016, Bak2017c, Bak2018a} and uncertain, piecewise-constant inputs \cite{Bak2017b}. Although this technique can compute very large systems---in some cases up to a billion states variables \cite{Bak2018a}---it cannot consider uncertain, arbitrarily time-varying inputs and requires a formally verified solver for linear systems. Even more recently, decomposition techniques have been developed to speed up reachability analysis of linear systems \cite{Bogomolov2018}; however, arbitrarily-varying inputs cannot be considered in \cite{Bogomolov2018}. Since one cannot exactly compute the reachable set of linear systems, except when all eigenvalues are real or imaginary \cite{Lafferriere1999a}, one typically demands tight over-approximations of reachable sets. 


\paragraph{Further use for nonlinear systems} Reachable set computations of linear systems are often embedded in algorithms for computing reachable sets of nonlinear systems, either by partitioning the state space into conservatively linearized regions \cite{Asarin2007} or by conservatively and continuously linearizing a nonlinear system along its center trajectory \cite{Althoff2008c, Dang2010}. With conservative linearization, we refer to techniques that compensate linearization errors by adding uncertainty, e.g., in the form of additive uncertain inputs. 

\paragraph{Order reduction techniques} For large-scale linear systems with uncertain time-varying inputs beyond $1000$ continuous state variables, however, even the most efficient reachability algorithms become too slow for practical use. One of the main reasons is that the exponential matrix $e^{At}$ of the system matrix $A$ in the linear system dynamics $\dot{x} = Ax + Bu$ may be unbearably time-consuming to compute \cite{Moler2003}. It should be noted that the dimension is not the only critical parameter; sparsity of $A$ and the sensitivity of the matrix exponential \cite[Sec.~2]{Moler2003} are also important parameters influencing the computation time. For this reason, methods have been developed to reduce the order of the investigated system, which are generally referred to as order reduction techniques. There exist a vast number of different techniques surveyed in e.g., \cite{Saad2003,Antoulas2005,Schilders2008}. Most order reduction techniques aim at achieving a similar input/output behavior when the system is initially in a steady state; they rely on the fact that large-scale systems often have a large number of state variables but a rather small number of input and output variables. A typical example is that of infinite-dimensional systems, which are spatially discretized, controlled by few actuators, and sensed by few sensors. Many applications of reachability analysis, however, such as formal verification, require that all or many state variables are accurately approximated as well. For instance, in a large power network, all voltages and frequencies have to stay within certain bounds. Subsequently, we review previous work on combining reachability analysis with order reduction techniques.  

\paragraph{Order reduction for reachability analysis} To the best knowledge of the author, the first work combining reachability analysis with order reduction techniques is \cite{Han2004}. There, simulating the solutions of all vertices of an initial set was required to guarantee an error bound for a set of initial states---this approach is exponential in the number of state variables when each variable is uncertain within an interval. At the time the approach in \cite{Han2004} was proposed, tools for reachability analysis had an exponential complexity as well, so that overall the computation time could be significantly reduced. However, modern tools such as SpaceEx \cite{Frehse2011}, Flow* \cite{Chen2013}, HyLAA \cite{Bak2017c}, XSpeed \cite{Gurung2016}, or CORA \cite{Althoff2015a} have a polynomial complexity, as demonstrated in \cite{Althoff2017b,Althoff2018c}, and thus would most certainly outperform the technique proposed in \cite{Han2004}, even without any order reduction. The same authors later combined reachability analysis with Krylov subspace approximation methods \cite{Han2006b}. The advantage of this technique is that it does not scale exponentially in the number of state variables; however, it can only handle linear systems with fixed input, which are also referred to as affine systems. Recently, the work from \cite{Han2004} was continued in \cite{Tran2017b}, which considers order reduction in the input/output sense for stable linear systems. In contrast, this work can also handle unstable systems and reconstruct the whole set of states, not just the outputs. For nonlinear systems, non-rigorous reduction techniques have been presented in \cite{Chou2017}; however, unlike in this work, the results are not formal.

\paragraph{Approximate bisimulation} Somewhat related to order reduction techniques is (approximate) bisimulation \cite{vanDerSchaft2004, Girard2011}, which basically shows that two systems have similar output behavior for the same inputs, while at the same time a relation $\mathcal{S}$ exists so that $(x(t),\tilde{x}(t)) \in \mathcal{S}$. However, even bisimulation techniques for linear systems do not scale to the system dimensions handled in this work, since they require solving linear matrix inequalities \cite{VanAntwerp2000} to find a bisimulation relation \cite{Girard2007a}. 

\paragraph{Contribution} This work extends the state of the art in reachability analysis of linear systems by proposing the first method with uncertain, time-varying inputs in the reduced Krylov subspace. This reduction makes it possible to compute reachable sets of systems with a number of continuous state variables that was previously infeasible. Please note that input/output order reduction techniques such as transformation, truncation, and projection (combination of transformation and truncation) cannot reconstruct the state, in contrast to the reduction presented in this work. However, we also improve computational efficiency when only the input/output behavior is of interest. Our proposed technique rigorously considers reduction errors. In contrast to \cite{Han2006b}, our approach a) does not rely on computing errors by the norm of the system matrix $\|A\|$ (see \eqref{eq:linearSystem}), which would quickly accumulate errors when $\|A\|$ is greater than $1$ (mostly true in practice; for the presented example of a bridge $\|A\|=4.6347\cdot 10^8$), b) does not require to enlarging the reachable set equally in all dimensions to account for errors, which can cause large over-approximations, and c) can be applied to zonotopes, which use a compact generator representation \cite[Theorem~7]{Althoff2009b}. Our approach is implemented in CORA \cite{Althoff2015a} and will be publicly released with the new CORA version. 

\paragraph{Organization} The paper is organized as follows: Sec.~\ref{sec:preliminaries} presents preliminaries from the areas of Krylov subspace approximation, set representation, and reachability analysis. The computation of reachable sets in the Krylov subspace is presented step by step: The homogeneous solution is described in Sec.~\ref{sec:KrylovReachHomogeneous} and the input solution in Sec.~\ref{sec:KrylovReachParticular}. Then, these are combined in Sec.~\ref{sec:propagation} to demonstrate the overall algorithm. We close with numerical examples in Sec.~\ref{sec:numericalExperiment} and the conclusions in Sec.~\ref{sec:conclusions}.



\section{Preliminaries} \label{sec:preliminaries}

Let us first recall some important basics required for this work: Krylov subspace approximation, representation of continuous sets in high-dimensional spaces, and computation of reachable sets of linear systems with uncertain inputs. 

\subsection{Krylov Subspace Approximation}

The main obstacle towards reachability analysis of large-scale linear systems is the evaluation of $e^{C}v$, where $C \in \mathbb{R}^{n \times n}$ and $v \in \mathbb{R}^n$. To compute $e^{C}v$ more efficiently, we introduce the Krylov subspace
\begin{equation*}
 \mathcal{K}_\xi = \mathtt{span}(v, Cv, \ldots, C^{\xi-1}v),
\end{equation*}
where $\mathtt{span}(\cdot)$ returns the linear span of a set of vectors and $\xi$ denotes the dimension of the subspace. Several possibilities have been developed for approximating $e^{C}v$ in the Krylov subspace \cite{Saad2003,Saad1992,Hochbruck1997}. In this work, we use the simplest approach, which is also one of the most popular: the Arnoldi algorithm, as presented in Alg.~\ref{alg:Arnoldi} (see \cite[Sec.~2.1]{Saad1992} and \cite[Alg.~1]{Sidje1998}). Please note that $w^*$ in Alg.~\ref{alg:Arnoldi} denotes the complex conjugate of $w$ and $\|.\|$ returns the Euclidean norm. A further reason for choosing the Arnoldi algorithm is that tight a-posteriori and a-prior error bounds exist \cite{Jia2015,Wang2017}.

\begin{algorithm}
\caption{Arnoldi iteration} \label{alg:Arnoldi}
\begin{algorithmic}[1]
	\Require $C$, $v$, max order $\xi$, tolerance $\mathtt{tol}$
	\Ensure $H$, $V$
	\State $v\gen{1} = v/\|v\|$
        \For{$k = 1\ldots \xi$}
	  \State $w = C v\gen{k}$
	  \For{$j = 1\ldots k$}
	    \State $h_{j,k} = w^* v\gen{j}$ 
	    \State $w := w - h_{j,k} v\gen{j}$
	  \EndFor
	  \State $h_{k+1,k} = \|w\|$
	  \If{$h_{k+1,k} \leq \mathtt{tol} \, \|C\|$} 
	    \State \textbf{happy-breakdown} \EndIf
	  \State $v\gen{k+1} = w/h_{k+1,k}$
        \EndFor
        \State $V = [v\gen{1}, v\gen{2}, \ldots, v\gen{\xi}]$
\end{algorithmic}
\end{algorithm}

The results of the Arnoldi iteration (using the stabilized Gram-Schmidt process) are an orthogonal basis $V = [v\gen{1}, v\gen{2}, \ldots, v\gen{\xi}]$ of the Krylov subspace $\mathcal{K}_\xi$ and the upper $\xi \times \xi$ Hessenberg matrix $H$ consisting of the elements $h_{ij}$ from Alg.~\ref{alg:Arnoldi}. Please note that we are numbering vectors with superscripted numbers in parentheses in order to avoid confusion with powers. Using $H$, $V$, and $e_1 = [1, \, 0, \, 0, \, \ldots, 0]^T$, the evaluation of $e^{C}v$ can be approximated as \cite[eq.~3-4]{Saad1992}:
\begin{align} 
 e^{C}v \approx & \|v\| V e^{H} e_1, \label{eq:expApprox} \\
 e^{Ct}v \approx & \|v\| V e^{Ht} e_1. \label{eq:expApprox_time} 
\end{align}
The following rigorous a-priori bound for the approximation error exists \cite[eq.~29]{Saad1992}: 
\begin{equation*} 
 \big\|e^{Ct}v - \|v\| V e^{Ht} e_1\big\| \leq 2\|v\| \frac{\|Ct\|^\xi \, e^{\|Ct\|}}{\xi!}.
\end{equation*}
A tighter a-priori bound is proposed in \cite{Wang2017} for different types of matrices (skew-Hermitian matrices, positive definite matrices, etc.). Even tighter results are obtained by using a-posteriori bound. Although they are computationally more expensive, they make it possible to obtain error bounds for long prediction horizons so that one typically does not have to recompute the Arnoldi iteration; this discussion is detailed in Sec.~\ref{sec:computationalComplexity}. In this work, we use the a-posteriori bound from \cite[Eq.~4.1]{Jia2015} so that
\begin{equation} \label{eq:KrylovErrorBound}
 \forall t \in[0, t_f]: \big\|e^{Ct}v - \|v\| V e^{Ht} e_1\big\| \leq \|v\| \overline{\epsilon}_\textrm{norm} \, t,
\end{equation}
where the computation of $\overline{\epsilon}_\textrm{norm}$ is described in detail in Appendix~\ref{app:errorBound}.

\subsection{Set Representation by Zonotopes}

As already summarized in the introduction, several set representations for reachability analysis of linear systems have been proposed. When considering uncertain inputs, zonotopes \cite{Girard2005} and support functions \cite{Girard2008b} demonstrate good performance. Recently, \cite{Althoff2016c} showed that combining zonotopes and support functions provides even better benefits: zonotopes can compute the solution more efficiently, while support functions can represent more general initial sets. Since most initial sets are multidimensional intervals (special case of zonotopes), we use zonotopes and neglect their combination with support functions to focus on the novel aspects of this work.  

\begin{definition}[Zonotope] \label{def:zonotope}
Given a center $c\in\mathbb{R}^{n}$ and so-called generators $g\gen{i}\in\mathbb{R}^{n}$, a zonotope is defined as
\begin{equation*}
	\mathcal{Z} := \Big\{ x\in\mathbb{R}^n \Big| x = c + \sum_{i=1}^{p} \beta_i g\gen{i}, \beta_i \in [-1,1] \Big\}.
\end{equation*}
\end{definition}
We write in short $\mathcal{Z}=(c,g\gen{1},\ldots,g\gen{p})$ and define the order of a zonotope as $o := \frac{p}{n}$, where $p$ is the number of generators. A zonotope can be seen as the Minkowski addition of line segments $[-1,1]g\gen{i}$, which provides an idea of how a zonotope is constructed; see Fig. \ref{fig:zonotope}.

\begin{figure}[htb]	

  \footnotesize
  \psfrag{c}[c][c]{$c$}						
  \psfrag{d}[c][c]{$g\gen{1}$}	
  \psfrag{e}[c][c]{$g\gen{2}$}	
  \psfrag{f}[c][c]{$g\gen{3}$}	
  \psfrag{g}[l][c]{$c \oplus g\gen{1}$}	
  \psfrag{h}[l][c]{$c \oplus g\gen{1} \oplus g\gen{2}$}	
  \psfrag{i}[l][c]{$c \oplus g\gen{1} \oplus g\gen{2} \oplus g\gen{3}$}
  \psfrag{j}[l][c]{\shortstack{construction \\ direction}}
  \psfrag{p}[c][c]{"$\oplus$"}
    \centering		
    \includegraphics[width=\columnwidth]{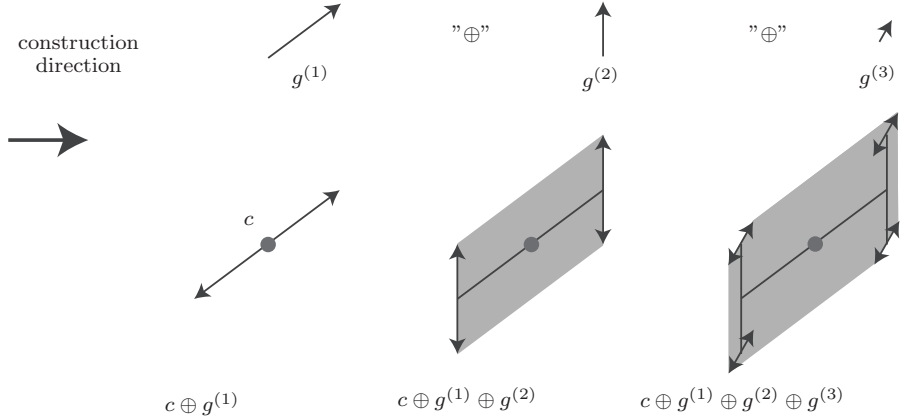}				
      \caption{Step-by-step construction of a two-dimensional zonotope.}
      \label{fig:zonotope}
\end{figure}

The iterative computation of reachable sets for linear systems requires set-based addition, which is often referred to as \textit{Minkowski addition} ($\mathcal{X} \oplus \mathcal{Y} :=  \{x+y | x\in \mathcal{X}, y\in \mathcal{Y} \}$), and set-based multiplication ($\mathcal{X} \otimes \mathcal{Y} :=  \{x \, y | x\in \mathcal{X}, y\in \mathcal{Y} \}$). Note that the symbol for set-based multiplication is often omitted for simplicity of notation, and that one or both operands can be singletons. The multiplication with a matrix $M\in \mathbb{R}^{o \times n}$ and the Minkowski addition of two zonotopes $\mathcal{Z}_1 = (c$, $g\gen{1}$, $\ldots$, $g\gen{p_1})$ and $\mathcal{Z}_2 = (d$, $h\gen{1}$, $\ldots$, $h\gen{p_2})$ are a direct consequence of the zonotope definition (see \cite{Kuehn1998}):
\begin{equation}\label{eq:zonoAdditionMultiplication}
\begin{split}
 \mathcal{Z}_1 \oplus \mathcal{Z}_2 & = (c + d, g\gen{1}, \ldots, g\gen{p_1}, h\gen{1}, \ldots, h\gen{p_2}), \\
 M\otimes \mathcal{Z}_1 & = (M \, c, M\, g\gen{1}, \ldots, M \, g\gen{p_1}).
\end{split}
\end{equation}
Also, the convex hull of $\mathcal{Z}_1$ and $\mathcal{Z}_2$ (both having equal order) is required (see \cite{Girard2005}):
\begin{equation}\label{eq:convexHull}
\begin{split}
 \mathtt{conv}(\mathcal{Z}_1,\mathcal{Z}_2) \subseteq & \frac{1}{2}( c+d, g\gen{1} + h\gen{1}, \ldots, g\gen{p_1} + h\gen{p_1}, \\
	    & \quad \, c-d, g\gen{1} - h\gen{1}, \ldots, g\gen{p_1} - h\gen{p_1}).
\end{split}
\end{equation}
For the multiplication of an interval matrix $\mathcal{M}$ with a zonotope, the matrix $\mathcal{M}$ is split into a real-valued matrix $M \in \mathbb{R}^{n \times n}$ and an interval matrix with bound $S \in \mathbb{R}^{n \times n}$, such that $\mathcal{M} = M \oplus [-S,S]$. After introducing $S_j$ as the $\th{j}$ row of $S$, the result is over-approximated as shown in \cite[Theorem~3.3]{Althoff2010a} by
\begin{equation*} 
\begin{split}
 	\mathcal{M}\mathcal{Z}_1 \subseteq & (M\mathcal{Z}_1 \oplus [-S,S]\mathcal{Z}_1) \\ \subseteq & (Mc_1,Mg\gen{1},\ldots,Mg\gen{p_1},h\gen{1},\ldots,h\gen{n}), \\
 	h\gen{i}_j= & \begin{cases} S_j(|c|+\sum_{k=1}^{p_1}|g|\gen{k}), \text{ for } i=j \\ 0, \text{ for } i\neq j \end{cases}.
\end{split}
\end{equation*}
Another important operation is the enclosure of a zonotope by an axis-aligned box (see \cite[Proposition~2.2]{Althoff2010a}):
\begin{equation} \label{eq:boxEnclosure}
\begin{split}
 \mathtt{box}(\mathcal{Z}) := & (c,\hat{h}\gen{1},\ldots,\hat{h}\gen{n}), \\
 \hat{h}\gen{i}_j= & \begin{cases} (\sum_{k=1}^{p_1}|g|\gen{k})_j, \text{ for } i=j \\ 0, \text{ for } i\neq j \end{cases}.
\end{split}
\end{equation}

We also require some new operations on zonotopes in the Krylov subspace, which are introduced later in Sec.~\ref{sec:KrylovReachHomogeneous}.

\subsection{Reachability Analysis} \label{sec:basicsReachabilityAnalysis}

Reachable set computations are typically performed iteratively for short time intervals 
\begin{equation} \label{eq:tau}
 \tau_k:=[t_k,t_{k+1}].
\end{equation}
In this work, constant-size time intervals $t_k := k\,\delta$ are used to focus on the main innovations, where $k\in\mathbb{N}$ is the time step and $\delta\in\mathbb{R}^+$ is referred to as the time increment. An extension to variable time increments is described in \cite{Frehse2011}. 

We recapitulate the reachability analysis of a linear differential inclusion 
\begin{equation} \label{eq:linearSystem}
 \dot{x} \in Ax(t) \oplus \mathcal{U},
\end{equation}
where $x\in\mathbb{R}^{n}$, $A\in\mathbb{R}^{n\times n}$, and $\mathcal{U}\subset\mathbb{R}^{n}$ is a set of uncertain inputs. Please note that this also includes the form $\dot{x} = Ax(t) + Bu(t)$, $u(t)\in \tilde{\mathcal{U}}$ often used in control theory, since one could easily choose $\mathcal{U} = B\tilde{\mathcal{U}}$. Let $\chi(t;x_0,u(\cdot))$ denote the solution to \eqref{eq:linearSystem} for an initial state $x(0)=x_0$ and the input trajectory $u(\cdot)$. For a set of initial states $\mathcal{X}_0 \subset \mathbb{R}^{n}$ and a set of possible input values $\mathcal{U}\subset \mathbb{R}^{m}$, the set of reachable states is
\begin{equation} \label{eq:defReachableSet}
\begin{split}
 \mathcal{R}^e([0,t_f]) := &\Big\{ \chi\left(t;x_0,u(\cdot)\right) \Big|
 x_0 \in \mathcal{R}(0), t\in[0,t_f], \\ 
&\hspace{60pt}\forall\tau\in[0,t] \, u(\tau)\in\mathcal{U} \Big\}.
\end{split}
\end{equation}
The superscript $e$ on $\mathcal{R}^e([0,t_f])$ denotes the exact reachable set, which cannot be computed for general linear systems \cite{Lafferriere2001}. Therefore, an over-approximation $\mathcal{R}([0,t_f])\supseteq \mathcal{R}^e([0,t_f])$ is computed as accurately as possible, while at the same time ensuring that the computations are efficient and scale well with the system dimension $n$.

For linear systems, the main task is to compute the reachable set of the first time interval $[0,\delta]$. Most of this paper will deal with computing the reachable set for the initial time interval since the propagation for later time intervals is rather simple, as shown later in Alg.~\ref{alg:reachsetLin}. We take advantage of the superposition principle of linear systems by computing the following reachable sets separately and later joining them together: the reachable set of the \underline{h}omogeneous solution $\mathcal{R}_h(t)$ ($u(\tau)=0$ in \eqref{eq:defReachableSet}), the reachable set of the \underline{p}articular solution $\mathcal{R}_p (t)$ due to the uncertain input $\mathcal{U}$ ($x_0=0$ in \eqref{eq:defReachableSet}), and the reachable set $\mathcal{R}_\epsilon$ correcting the initial assumption that trajectories are straight lines within $[0,\delta]$. According to \cite[Sec.~3.2]{Althoff2010a}, the reachable set for $[0,\delta]$ is computed as shown in Fig. \ref{fig:linReachOverview}:

\begin{enumerate}
 \item Starting from $\mathcal{X}_0$, compute the set of all homogeneous solutions $\mathcal{R}_h(\delta)$.
 \item Obtain the convex hull of $\mathcal{X}_0$ and $\mathcal{R}_h(\delta)$ to approximate the reachable set for the time interval $[0,\delta]$.
 \item Compute $\mathcal{R}([0,\delta]):=\bigcup_{t\in [0,\delta]} \mathcal{R}(t)$ by considering uncertain inputs by adding $\mathcal{R}_p(\delta)$ and accounting for the curvature of trajectories by adding $\mathcal{R}_\epsilon$.
\end{enumerate}

\begin{figure}[htb]
    \centering	
    \footnotesize
				
      \psfrag{a}[c][c]{ $\mathcal{R}_h(\delta)$}									
      \psfrag{b}[c][c]{ $\mathcal{X}_0$}			
      \psfrag{c}[c][c]{\shortstack{convex hull of \\
			$\mathcal{X}_0$, $\mathcal{R}_h(\delta)$}}
      \psfrag{d}[r][c]{ $\mathcal{R}([0,\delta])$}						
      \psfrag{e}[c][c]{\ding{192}}							
      \psfrag{f}[c][c]{\ding{193}}
      \psfrag{g}[c][c]{\ding{194}}						
      \psfrag{h}[l][c]{\shortstack{enlarge- \\ment}}					
      \includegraphics[width=\columnwidth]{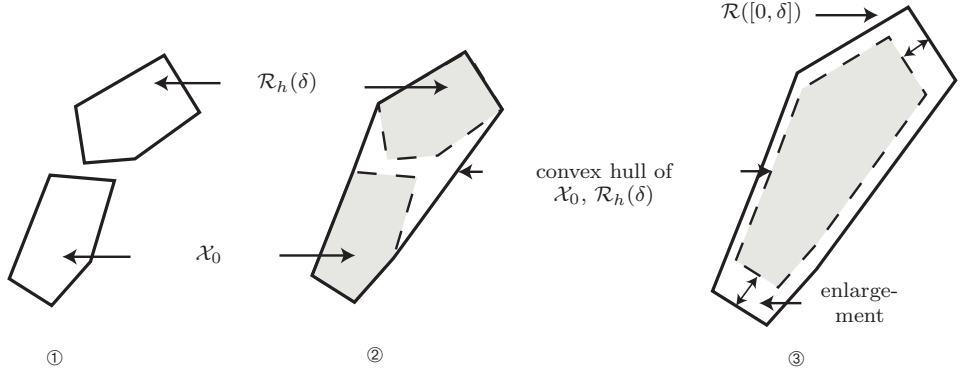}
      \caption{Steps for computing the reachable set of a linear system for the first time interval.}
      \label{fig:linReachOverview}
\end{figure}

\section{Homogeneous Solution in the Krylov Subspace} \label{sec:KrylovReachHomogeneous}

In this section, the basic idea of computing reachable sets as presented in Sec.~\ref{sec:basicsReachabilityAnalysis} is extended so that reachable sets can be computed in the Krylov subspace. We first present new techniques for the homogeneous solution of points in time and time intervals. Subsequently, we consider for the first time how reachable sets can be computed for arbitrarily changing input trajectories within the Krylov subspace.

\subsection{Solution for a Point in Time}

The well-known homogeneous solution of a linear time-invariant system with initial state $x_0$ is
\begin{equation*}
 x_h(t) = e^{At} x_0.
\end{equation*}
We can bound the exact solution using the lemma below. For that lemma and subsequent derivations we introduce $\mathbf{[-1,1]}^n$ as an $n$-dimensional vector whose entries are intervals $[-1,1]$. Analogously, we write $\mathbf{[-1,1]}^{n \times m}$ to represent an $n \times m$ matrix whose entries are intervals $[-1,1]$.
\begin{lemma}[Single state homogeneous solution] \label{thm:singleStateHomSol}
After obtaining $V$ and $H$ from Alg.~\ref{alg:Arnoldi} with inputs $C=A$, $v=x_0$, we can bound the homogeneous solution $x_h(t) = e^{At} x_0$ by
 \begin{equation} \label{eq:initialStateSolution}
\begin{split}
 x_h(t) & \in \|x_0\| V e^{H \, t} e_1 \oplus \mathbf{[-1,1]}^n \|x_0 \| \epsilon_\textrm{norm} \, t\\
 & = \hat{x}_h (t) \oplus \mathcal{E}_\mathtt{red}(t,x_0),
\end{split}
\end{equation}
where
\begin{equation} \label{eq:approxSolutionAndError}
\begin{split}
 \hat{x}_h (t) &:= \|x_0\| V e^{H \, t} e_1 \\ 
 \mathcal{E}_\mathtt{red}(t,x_0) &= \mathbf{[-1,1]}^n \|x_0 \| \epsilon_\textrm{norm} \, t
\end{split}
\end{equation}
and $\epsilon_\textrm{norm}$ is computed as described in Appendix~\ref{app:errorBound}.
\end{lemma}
\begin{proof}
 The lemma directly follows from enlarging the approximate solution in \eqref{eq:expApprox} by the error $\|x_0 \| \epsilon_\textrm{norm}\, t$ from \eqref{eq:KrylovErrorBound}. Since we have for a vector $a \in \mathbb{R}^n$ and a scalar $b \in \mathbb{R}$ that $\|a\| < b \Rightarrow a \in \mathbf{[-1,1]}^n b$, one obtains from \eqref{eq:KrylovErrorBound} that
 \begin{equation*} 
  e^{A\, t}x_0 - \|x_0\| V e^{H \, t} e_1 \in \mathbf{[-1,1]}^n \|x_0\| \epsilon_\textrm{norm} \, t,
 \end{equation*}
 which proves the lemma. 
\end{proof}

For reachability analysis, one has to compute the homogeneous solution for a set of initial states. Replacing the single initial state in Lemma~\ref{thm:singleStateHomSol} by a set of initial states $\mathcal{X}_0$ is not trivial since the matrices $V$ and $H$ depend on each initial state $x_0 \in \mathcal{X}_0$. To indicate this dependency, we write $V(x_0)$ and $H(x_0)$ from now on to stress that those matrices have been obtained from the state $x_0$. Since $\mathcal{X}_0$ is represented as a zonotope in this work, the homogeneous reachable set can be computed by the following theorem:

\begin{theorem}[Homogeneous solution for a point in time] \label{thm:homSolution}
The reachable set of the homogeneous solution $\mathcal{R}_h(t) := e^{A\, t} \mathcal{X}_0$ for the initial zonotope $\mathcal{X}_0 = (c,g\gen{1},\ldots,g\gen{p})$ can be over-approximated by the zonotope
\begin{equation*}
\begin{split}
  \mathcal{R}_h(t) & \subseteq (\hat{c}, \hat{g}\gen{1},\ldots,\hat{g}\gen{p}) \oplus \mathcal{R}_{h,err}, \\
  \hat{c} & = \|c \| V(c) e^{H(c) \, t} \, e_1, \\
  \hat{g}\gen{i} & = \|g\gen{i} \| V(g\gen{i}) e^{H(g\gen{i}) \, t} \, e_1, \\
  \mathcal{R}_{h,err} & = \mathbf{[-1,1]}^n \Big(\|c \| + \sum_{i=1}^p \|g\gen{i} \|\Big) \epsilon_\textrm{norm} \, t.
\end{split}
\end{equation*}
\end{theorem}
\begin{proof}
 Inserting the definition of a zonotope (Def.~\ref{def:zonotope}) into $\mathcal{R}_h(t) = e^{A\, t} \mathcal{X}_0$ and using \eqref{eq:zonoAdditionMultiplication}, we obtain
 \begin{equation*}
  \mathcal{R}_h(t) = e^{A\, t} c \oplus \bigoplus_{i=1}^p [-1,1] e^{A\, t} g\gen{i}.
 \end{equation*}
 Using Lemma~\ref{thm:singleStateHomSol} yields
  \begin{equation*}
  \begin{split}
   \mathcal{R}_h(t) \subseteq & \|c\| V(c) e^{H(c) \, t} e_1 \oplus \mathbf{[-1,1]}^n \|c \| \epsilon_\textrm{norm} \, t \\
   & \oplus \bigoplus_{i=1}^p \Big([-1,1] \|g\gen{i}\| V(g\gen{i}) e^{H(g\gen{i})\, t} e_1 \\
   & \oplus \mathbf{[-1,1]}^n \|g\gen{i} \| \epsilon_\textrm{norm} \, t \Big) \\
   =&  \|c\| V(c) e^{H(c) \, t} e_1 \\
   & \oplus \bigoplus_{i=1}^p [-1,1] \|g\gen{i}\| V(g\gen{i}) e^{H(g\gen{i}) \, t} e_1 \\
   & \oplus \underbrace{\mathbf{[-1,1]}^n \Big(\|c \| + \sum_{i=1}^p \|g\gen{i} \|\Big) \epsilon_\textrm{norm} \, t}_{= \mathcal{R}_{h,err}}.
  \end{split}
 \end{equation*}
 This results in the zonotope of the theorem. Since $\mathcal{R}_{h,err}$ is a multidimensional interval, and thus also a zonotope, the addition of the two zonotopes $\mathcal{R}_{h,err}$ and $(\hat{c}, \hat{g}\gen{1},\ldots,\hat{g}\gen{p})$ results in the zonotope $\mathcal{R}_h(t)$.
\end{proof}

We propose two different representations of $\mathcal{R}_{h,err}$. The first one uses the equivalence
\begin{equation} \label{eq:intervalVectorEquivalence}
 \begin{split}
 \mathbf{[-1,1]}^n \underbrace{\Big(\|c \| + \sum_{i=1}^p \|g\gen{i} \|\Big) \epsilon_\textrm{norm} \, t}_{=: \hat{\epsilon}_\textrm{norm} \, t} = \bigoplus_{i=1}^n [-1,1] e_i \, \hat{\epsilon}_\textrm{norm} \, t,
 \end{split}
\end{equation}
where $e_i$ is a unit vector with the $i^{th}$ element being $1$ and all others $0$. Let us also introduce $\mathbf{0}^n$ and $\mathbf{1}^n$ as an n-dimensional vector of zeros and ones, respectively. Using \eqref{eq:intervalVectorEquivalence}, we can write $\mathcal{R}_{h,err} = (\mathbf{0}^n, \hat{h}\gen{1},\ldots,\hat{h}\gen{n})$ with
\begin{equation*}
 \hat{h}\gen{i}_j= \begin{cases} \hat{\epsilon}_\textrm{norm} \, t, \text{ for } i=j \\ 0, \text{ for } i\neq j \end{cases},
\end{equation*}
which makes it possible to obtain (see Theorem~\ref{thm:homSolution})
\begin{equation*}
\begin{split}
 \mathcal{R}_h(t) \subseteq & (\hat{c}, \hat{g}\gen{1},\ldots,\hat{g}\gen{p}) \oplus \mathcal{R}_{h,err} \\
 = & (\hat{c}, \hat{g}\gen{1},\ldots,\hat{g}\gen{p},\hat{h}\gen{1},\ldots,\hat{h}\gen{n}).
\end{split}
\end{equation*}
To avoid adding new generators $\hat{h}\gen{1},\ldots,\hat{h}\gen{n}$, one can also bound $\mathcal{R}_{h,err}$ by an interval vector as presented in the next lemma. 

\begin{lemma}[Interval vector multiplication]
 Using $[\underline{b}, \overline{b}] = b_c \oplus [-b_\Delta, b_\Delta]$ ($b_\Delta \geq 0$) we can provide the bound
\begin{equation*}
\begin{split}
 e^{A\,t}[\underline{b}, \overline{b}] \subseteq& \|b_c\| V(b_c) e^{H(b_c)\, t} e_1 \oplus [-\mu, \mu],
\end{split}
\end{equation*}
where
\begin{equation*}
 \mu = \|b_\Delta\|\bar{V}(b_\Delta) e^{\bar{H}(b_\Delta)\, t} e_1 + \mathbf{1}^n (\|b_\Delta\|\bar{\epsilon}_\textrm{norm} \, t + \|b_c \|\epsilon_\textrm{norm} \, t)
\end{equation*}
and $\bar{V}$, $\bar{H}$ are obtained from Alg.~\ref{alg:Arnoldi} and $\bar{\epsilon}_\textrm{norm} \, t$ from Appendix~\ref{app:errorBound} using $C=|A|$. Please note that the absolute values are computed elementwise, i.e., $|A|_{i,j} = |A_{i,j}|$.
\end{lemma}
\begin{proof}
We first show that $e^{A\,t}[\underline{b}, \overline{b}]$ can be over-approximated by 
\begin{equation*}
\begin{split}
 e^{A\,t}[\underline{b}, \overline{b}] =& e^{A\,t}(b_c \oplus [-b_\Delta, b_\Delta]) \\
			       \subseteq& e^{A\,t} b_c \oplus e^{A\,t}[-b_\Delta, b_\Delta] \\
			       \subseteq& e^{A\,t} b_c \oplus [-|e^{A\,t}|b_\Delta, |e^{A\,t}|b_\Delta] \\
			       \subseteq& e^{A\,t} b_c \oplus [-e^{|A|t}b_\Delta, e^{|A|t}b_\Delta],
\end{split}
\end{equation*}
where the over-approximation achieved in the last line directly follows from the Taylor series of $e^{A\,t}$ and $A_{i,j} \leq C_{i,j}$. Using the above result and Lemma~\ref{thm:singleStateHomSol}, we obtain
\begin{equation} \label{eq:intervalComputation}
\begin{split}
 & e^{A\,t} b_c \oplus [-e^{|A|t}b_\Delta, e^{|A|t}b_\Delta] \\
 =& \|b_c\| V(b_c) e^{H(b_c)\,t} e_1 \oplus \mathbf{[-1,1]}^n \|b_c \| \epsilon_\textrm{norm} \, t \oplus [-\gamma, \gamma],
\end{split}
\end{equation}
where
\begin{equation*}
 \gamma = \|b_\Delta\|\bar{V}(b_\Delta) e^{\bar{H}(b_\Delta)\,t} e_1 + \mathbf{1}^n \| b_\Delta \| \bar{\epsilon}_\textrm{norm} \, t.
\end{equation*}
The result in \eqref{eq:intervalComputation} can be simplified to 
\begin{equation*}
\begin{split}
 \|b_c\|V(b_c) e^{H(b_c)\,t} e_1 \oplus [-\mu, \mu],
\end{split}
\end{equation*}
with $\mu$ as in the lemma.
\end{proof}

\subsection{Solution for a Time Interval}

In the previous subsection, we over-approximated the homogeneous solution for points in time. This subsection over-approximates $x_h (t)$ for a time interval $[0,\delta]$. Since $x_h (t) \in \hat{x}_h(t) \oplus \mathcal{E}_\mathtt{red}(t,x_0)$ as shown in \eqref{eq:initialStateSolution}, where $\mathcal{E}_\mathtt{red}(t,x_0)$ is monotonically increasing according to \eqref{eq:approxSolutionAndError}, we have that $\forall t\in[0,\delta]: \, \mathcal{E}_\mathtt{red}(t,x_0) \subseteq \mathcal{E}_\mathtt{red}(\delta,x_0)$. It remains to approximate $\hat{x}_h(t)$ within time intervals:
\begin{equation*}
\begin{split}
  \forall t\in[0,\delta]: \quad \hat{x}_h (t) \approx x_0+\frac{t}{\delta}(\hat{x}_h(\delta)-x_0).
\end{split}
\end{equation*}
To consider the error of this approximation, we use a finite Taylor expansion of the exponential matrix of $\eta^{\text{th}}$ order with error matrix $E$ (see \cite[Prop.~2]{Althoff2011a}):
\begin{equation}\label{eq:taylorSeries}
\begin{split}
	e^{H(x_0) \, t} =& \sum_{i=0}^{\eta} \frac{1}{i!}(H(x_0) \, t)^i + E(t, x_0),
\end{split}
\end{equation}
where $E$ is enclosed by an interval matrix:
\begin{gather}\label{eq:matrixExponentialRemainder}
 E(t, x_0) \in \mathcal{E}(t, x_0)  = [-W(t, x_0),W(t, x_0)], \\
 W(t, x_0)= \sum_{i=\eta+1}^{\infty} \frac{t^i}{i!}|H(x_0)|^i = e^{|H(x_0)|t} - \sum_{i=0}^{\eta} \frac{t^i}{i!}|H(x_0)|^i.\notag
\end{gather}
We are proposing an enclosure of the error $x_h (t) - \hat{x}_h (t)$, which is multiplicative with the initial state, since large initial states result in larger errors for a given time horizon.
\begin{lemma}[Correction interval matrix $\mathcal{F}$]\label{thm:correctionMatrix}
The interval matrix 
\begin{equation*}
	\mathcal{F}(x_0) = \sum_{i=2}^{\eta} [(i^{\frac{-i}{i-1}}-i^{\frac{-1}{i-1}})\delta^i,0]\frac{H^i(x_0)}{i!} \oplus \mathcal{E}(\delta,x_0),
\end{equation*}
with $\mathcal{E}(\delta,x_0)$ and $\eta$ according to \eqref{eq:matrixExponentialRemainder} ensures the enclosure of the exact solution:
\begin{equation} \label{eq:timeIntervalEnlargement}
\begin{split}
  \forall t\in[0,\delta]: \quad \hat{x}_h (t) \, \in \, & x_0+\frac{t}{\delta}(\hat{x}_h(\delta)-x_0) \\
  & \oplus \|x_0\| V(x_0) \mathcal{F}(x_0) e_1.
\end{split}
\end{equation}
\end{lemma}

\begin{proof}
Let us start by rearranging \eqref{eq:timeIntervalEnlargement}: 
\begin{equation*}
 \hat{x}_h (t) - x_0 - \frac{t}{\delta}(\hat{x}_h(\delta)-x_0) \\
  \in \|x_0\| V(x_0) \mathcal{F}(x_0) e_1.
\end{equation*}
After replacing $\hat{x}_h(t)$ by $\|x_0\| V(x_0) e^{H(x_0)\, t} e_1$ from \eqref{eq:approxSolutionAndError} and using $\|x_0\| V(x_0) e_1 = x_0$ (this follows from the fact that the first column of $V(x_0)$ is $x_0/\|x_0\|$ according to Alg.~\ref{alg:Arnoldi}), the above inclusion becomes 
\begin{equation*}
 \begin{split}
  & \|x_0\| V(x_0) e^{H(x_0)\, t} e_1 - \|x_0\| V(x_0) e_1 \\
  & -\frac{t}{\delta}\Big(\|x_0\| V(x_0) e^{H(x_0)\, \delta} e_1 - \|x_0\| V(x_0) e_1 \Big) \\
  \in & \|x_0\| V(x_0) \mathcal{F}(x_0) e_1.
   \end{split}
\end{equation*}
Next, we use the result $A(B+C)D = ABD + ACD$ from linear algebra to simplify the above equation to
\begin{equation*}
 \begin{split}
  & \|x_0\| V(x_0) \bigg( e^{H(x_0)\, t} - I - \frac{t}{\delta}\Big(e^{H(x_0)\, \delta} - I \Big)\bigg)e_1 \\
  \in & \|x_0\| V(x_0) \mathcal{F}(x_0) e_1. \\
 \end{split}
\end{equation*}
By comparing the inner part $\square$ of $\|x_0\| V(x_0) \square e_1$, one obtains
\begin{equation*}
	\forall t\in[0,\delta]: \, e^{H(x_0) \, t} -I - \frac{t}{\delta}(e^{H(x_0)\, \delta}-I) \in \mathcal{F}(x_0).
\end{equation*}
Substituting $e^{H(x_0) \, t}$ by \eqref{eq:taylorSeries} and canceling linear terms yields
\begin{equation*}
	\sum_{i=2}^{\eta} (t^i-t\, \delta^{i-1})\frac{1}{i!} H^i(x_0) + E(t,x_0) - \frac{t}{\delta} E(\delta,x_0) \in \mathcal{F}(x_0).
\end{equation*}
The interval of $t^i-t\, \delta^{i-1}$ for $t\in[0,\delta]$ is obtained exactly by computing the minimum and maximum of $t^i-t\,\delta^{i-1}$ for which only one extreme value exists since the second derivative is nonnegative: $\frac{d}{dt}(t^i-t\,\delta^{i-1})=0 \Rightarrow t_{min}=i^{-\frac{1}{i-1}}\delta$. This means that the maximum values are to be found at the borders of $t\in[0,\delta]$, which are both $0$ for $t=0$ and $t=\delta$. Thus, 
\begin{equation*}
\begin{split}
	\{t^i-t\,\delta^{i-1} | t \in [0,\delta] \} &= [t_{min}^i-t_{min}\delta^{i-1}, 0] \\
	&= [(i^{\frac{-i}{i-1}}-i^{\frac{-1}{i-1}})\delta^i,0].
\end{split}
\end{equation*}
It remains to bound $E(t,x_0) - \frac{t}{\delta} E(\delta,x_0)$ for $t\in[0,\delta]$. Using \eqref{eq:matrixExponentialRemainder} we have that
\begin{equation*}
 \begin{split}
  & E(t,x_0) - \frac{t}{\delta} E(\delta,x_0) \\
  \leq & \bigg|E(t,x_0) - \frac{t}{\delta} E(\delta,x_0)\bigg| \\
  \overset{\eqref{eq:matrixExponentialRemainder}}{=} & \bigg|\sum_{i=\eta+1}^{\infty} \frac{t^i}{i!}|H(x_0)|^i - \frac{t}{\delta}\sum_{i=\eta+1}^{\infty} \frac{\delta^i}{i!}|H(x_0)|^i\bigg| \\
  = & \bigg|\sum_{i=\eta+1}^{\infty} (t^i - t\delta^{i-1})\frac{|H(x_0)|^i}{i!} \bigg| \\
  = & \sum_{i=\eta+1}^{\infty} \Big|t^i - t\delta^{i-1}\Big|\frac{|H(x_0)|^i}{i!}  \\
  \leq & \sum_{i=\eta+1}^{\infty} \delta^i \frac{|H(x_0)|^i}{i!} = E(\delta,x_0).
 \end{split}
\end{equation*}
Thus, $\forall t\in[0,\delta]: \, E(t,x_0) - \frac{t}{\delta} E(\delta,x_0) \in \mathcal{E}(\delta,x_0)$, which completes the proof.
\end{proof}

When the initial state $x_0$ is substituted by the set of initial states $\mathcal{X}_0$, Lemma~\ref{thm:correctionMatrix} can be generalized as described in the subsequent theorem.
\begin{theorem}[Solution for time intervals] \label{thm:timeInterval}
 The reachable set for $t\in[0,\delta]$ is over-approximated by
 \begin{equation*}
  \begin{split}
    \mathcal{R}_h([0,\delta]):= \bigcup_{t\in[0,\delta]}\mathcal{R}_h(t) =& \mathtt{conv}\big(\mathcal{X}_0,\mathcal{R}_h(\delta)\big) \oplus \mathcal{N},
  \end{split}
 \end{equation*}
 where 
 \begin{equation*}
  \mathcal{N} = \|c\| V(c) \mathcal{F}(c) e_1 \oplus \bigoplus_{i=1}^p \|g\gen{i}\| V(g\gen{i}) \mathcal{F}(g\gen{i}) e_1.
 \end{equation*}
\end{theorem}
\begin{proof}
 Let us start by formulating the reachable set for $t\in[0,\delta]$ using $x_h(t) \in \hat{x}_h (t) \oplus \mathcal{E}_\mathtt{red}(t, x_0)$ from \eqref{eq:initialStateSolution} and $\hat{x}_h (t)$ from \eqref{eq:timeIntervalEnlargement}; we also use $\forall t \in [0,\delta]: \mathcal{E}_\mathtt{red}(t, x_0) \subseteq \mathcal{E}_\mathtt{red}(\delta, x_0)$:
 \begin{equation} \label{eq:trivialSolutionTimeInterval}
  \begin{split}
   \mathcal{R}_h(t) \subseteq & \Big\{x_0+\frac{t}{\delta}(\hat{x}_h(\delta)-x_0)\Big| x_0 \in \mathcal{X}_0 \Big\} \\
   & \oplus \bigcup_{x_0 \in \mathcal{X}_0} \|x_0\| V(x_0) \mathcal{F}(x_0) e_1 \\
   & \oplus \frac{t}{\delta} \bigcup_{x_0 \in \mathcal{X}_0} \mathcal{E}_\mathtt{red}(\delta, x_0).
  \end{split}
 \end{equation}
 Since 
 \begin{equation*}
 \begin{split}
  \forall t\in[0,\delta]: & \Big\{x_0+\frac{t}{\delta}(\hat{x}_h(\delta)-x_0) \Big| x_0 \in \mathcal{X}_0 \Big\} \\
  & \oplus \frac{t}{\delta} \bigcup_{x_0 \in \mathcal{X}_0} \mathcal{E}_\mathtt{red}(\delta, x_0) \\
  & \subseteq \mathtt{conv}(\mathcal{X}_0,\mathcal{R}_h(\delta))
 \end{split}
 \end{equation*}
 we can simplify \eqref{eq:trivialSolutionTimeInterval} to
 \begin{equation*} 
  \begin{split}
   \mathcal{R}_h([0,\delta]) \subseteq & \mathtt{conv}(\mathcal{X}_0,\mathcal{R}_h(\delta))  \\
   & \oplus \bigcup_{x_0 \in \mathcal{X}_0} \|x_0\| V(x_0) \mathcal{F}(x_0) e_1 \\
   \subseteq & \mathtt{conv}(\mathcal{X}_0,\mathcal{R}_h(\delta)) \oplus \hat{\mathcal{N}},
  \end{split}
 \end{equation*}
 where $\hat{\mathcal{N}} := \bigcup_{x_0 \in \mathcal{X}_0} \|x_0\| V(x_0) \mathcal{F}(x_0) e_1$. It remains to over-approximate $\hat{\mathcal{N}}$ when the initial set is a zonotope:
 \begin{equation*}
 \begin{split}
  & \bigcup_{x_0 \in \mathcal{X}_0} \|x_0\| V(x_0) \mathcal{F}(x_0) e_1 \\
  \subseteq & \underbrace{\|c\| V(c) \mathcal{F}(c) e_1}_{=: \mathcal{N}\gen{0}} \oplus \bigoplus_{i=1}^p \underbrace{[-1,1] \|g\gen{i}\| V(g\gen{i}) \mathcal{F}(g\gen{i}) e_1}_{= \|g\gen{i}\| V(g\gen{i}) \mathcal{F}(g\gen{i}) e_1 =: \mathcal{N}\gen{i}} =: \mathcal{N},
 \end{split} 
 \end{equation*}
 and the interval vectors $\mathcal{N}\gen{i}$ are added using standard interval arithmetic \cite{Jaulin2006}. Thus, $\mathcal{N}$ is an interval vector, which is converted to a zonotope and added to the convex hull. The scalar interval $[-1,1]$ can be moved right in front of $\mathcal{F}(g\gen{i})$ and since all entries of $\mathcal{F}(g\gen{i})$ have symmetric bounds, we have that $[-1,1]\mathcal{F}(g\gen{i}) = \mathcal{F}(g\gen{i})$ so that $[-1,1]$ can be removed.
\end{proof}

Next, we derive the set of solutions due to uncertain inputs.

\section{Input Solution in the Krylov Subspace} \label{sec:KrylovReachParticular}

In this section, we obtain for the first time the set of input solutions for uncertain, time-varying inputs in the Krylov subspace. We first consider over-approximations for solutions of a single constant input. Next, we generalize this result to uncertain but constant inputs. Finally, we derive an over-approximation for arbitrarily-varying and uncertain inputs. The first lemma provides the over-approximation for a constant, known input. 
\begin{lemma}[Krylov error of constant input solution] \label{thm:KrylovErrorInput}
After obtaining $\tilde{V}(u)$, $\tilde{H}(u)$ from Alg.~\ref{alg:Arnoldi} and $\tilde{\epsilon}_\textrm{norm} \, \delta$ from Appendix~\ref{app:errorBound} with inputs 
\begin{equation*}
 C = \begin{bmatrix} A & u \\ \mathbf{0}^{1 \times n} & 0 \end{bmatrix}, \quad v = \begin{bmatrix} \mathbf{0}^n \\ 1 \end{bmatrix},
\end{equation*}
we can bound the particular solution (aka input solution) for constant inputs 
 \begin{equation} \label{eq:constInputSolution}
  x_{p,\mathtt{const}}(\delta) = \int_0^\delta e^{A (\delta - t)} \mathtt{d}t \, u = \int_0^\delta e^{A t} \mathtt{d}t \, u
 \end{equation}
by
 \begin{equation*} 
\begin{split}
 x_{p,\mathtt{const}}(\delta) & \in P \, \tilde{V}(u) e^{\tilde{H}(u) \, \delta} e_1 \oplus \mathbf{[-1,1]}^n \tilde{\epsilon}_\textrm{norm} \, \delta \\
 & = \tilde{x}_p (\delta) \oplus \tilde{\mathcal{E}}_\mathtt{red} (\delta),
\end{split}
\end{equation*}
where
\begin{equation} \label{eq:approxInputSolutionAndError}
\begin{split}
 P &= [I, \mathbf{0}^n], \qquad \text{($I$ is the $n\times n$ identity matrix)} \\
 \tilde{x}_p (\delta) &= P\, \tilde{V}(u) e^{\tilde{H}(u) \, \delta} e_1  \\
 \tilde{\mathcal{E}}_\mathtt{red} (\delta) &= \mathbf{[-1,1]}^n \tilde{\epsilon}_\textrm{norm} \, \delta.
\end{split}
\end{equation}
\end{lemma}

\begin{proof}
 We first rewrite the solution of \eqref{eq:constInputSolution} to the differential equation below as also performed in \cite[Sec.~5]{Sidje1998}:
 \begin{equation*} 
  \underbrace{\begin{bmatrix} \dot{x}_{p,\mathtt{const}} \\ \dot{\check{x}} \end{bmatrix}}_{\dot{\tilde{x}}} = \underbrace{\begin{bmatrix} A & u \\ \mathbf{0}^{1 \times n} & 0 \end{bmatrix}}_{=:\tilde{A}(u)} \underbrace{\begin{bmatrix} x_{p,\mathtt{const}} \\ \check{x} \end{bmatrix}}_{\tilde{x}}, \quad \tilde{x}(0) = \tilde{x}_0 = \begin{bmatrix} \mathbf{0}^n \\ 1 \end{bmatrix}.
 \end{equation*}
 After inserting the projection matrix $P=[I, \mathbf{0}^n]$ (see \eqref{eq:approxInputSolutionAndError}), we can write
 \begin{equation*} 
  x_{p,\mathtt{const}}(t) = P \, e^{\tilde{A}(u)t} \tilde{x}_0.
 \end{equation*}
 This makes it possible to reformulate the Krylov error of the input solution:
 \begin{equation*}
  \begin{split}
  & \|\underbrace{P \, e^{\tilde{A}(u)\, t} \tilde{x}_0}_{x_{p,\mathtt{const}}(t)} - \underbrace{P\, \|\tilde{x}_0\| \tilde{V}(u) e^{\tilde{H}(u)\, t} e_1}_{\text{Krylov approximation of }x_{p,\mathtt{const}}(t)} \| \\
  \leq & \|P\| \, \|e^{\tilde{A}(u)\, t} \tilde{x}_0 - \|\tilde{x}_0\|\tilde{V}(u) e^{\tilde{H}(u)\, t} e_1 \|, \, \, \, (\|P\|=\| \tilde{x}_0\| =1)\\
  = & \|e^{\tilde{A}(u)\, t} \tilde{x}_0 - \tilde{V}(u) e^{\tilde{H}(u)\, t} e_1 \| \leq \tilde{\epsilon}_\textrm{norm} \, t
  \end{split}
  \end{equation*}
  and $\tilde{\epsilon}_\textrm{norm} \, t$ is obtained as in Appendix~\ref{app:errorBound}.
\end{proof}

In order to generalize the previous results to arbitrarily-varying inputs, we require the following corollary on the input solution for general time bounds:
\begin{corollary}[Input solution for general time bounds] \label{thm:KrylovErrorInput_generalTimeBounds}
 The partial input solution for 
 \begin{equation*}
  x_{p,\mathtt{const}}([t_0,t_e]) :=  \int_{t_0}^{t_e} e^{A (\delta - t)} \mathtt{d}t \, u,
 \end{equation*}
 where $0\leq t_0 \leq t_e \leq \delta$, can be over-approximated as
\begin{equation*}
\begin{split}
x_{p,\mathtt{const}}([t_0,t_e]) \in & P\, \tilde{V}(u) (e^{\tilde{H}(u) \, (\delta - t_0)} - e^{\tilde{H}(u) \, (\delta - t_e)}) e_1 \\
  & \oplus \mathbf{[-1,1]}^n \tilde{\epsilon}_\textrm{norm} \, (t_e - t_0).
\end{split}
\end{equation*}
\end{corollary}
\begin{proof}
 Let us rewrite $x_{p,\mathtt{const}}([t_0,t_e])$ as 
 \begin{equation*} 
 \begin{split}
  x_{p,\mathtt{const}}([t_0,t_e]) :=& \int_{t_0}^{t_e} e^{A (\delta - t)} \mathtt{d}t \, u = \int_{\delta - t_e}^{\delta - t_0} e^{A t} \mathtt{d}t \, u \\
  =& \Big(\int_{0}^{\delta - t_0} e^{A t} \mathtt{d}t - \int_{0}^{\delta - t_e} e^{A t} \mathtt{d}t \Big) u.
 \end{split}
 \end{equation*}
 After applying Lemma~\ref{thm:KrylovErrorInput}, we obtain
 \begin{equation*}
 \begin{split}
  & x_{p,\mathtt{const}}([t_0,t_e]) \\
  = & P\, \tilde{V}(u) e^{\tilde{H}(u) \, (\delta - t_0)} e_1 - P\, \tilde{V}(u) e^{\tilde{H}(u) \, (\delta - t_e)}e_1 \\
  & \oplus \mathbf{[-1,1]}^n \tilde{\epsilon}_\textrm{norm} \, (\delta - t_0 - (\delta - t_e)) \\
  =& P\, \tilde{V}(u) (e^{\tilde{H}(u) \, (\delta - t_0)} - e^{\tilde{H}(u) \, (\delta - t_e)}) e_1 \\
  & \oplus \mathbf{[-1,1]}^n \tilde{\epsilon}_\textrm{norm} \, (t_e - t_0).
 \end{split} 
 \end{equation*}
\end{proof}

Next, we over-approximate the solution for constant inputs which are uncertain within the set $\mathcal{U}$. In many cases, it is desired to have the solution for constant inputs, e.g., if a control system with zero-order hold is considered.
\begin{theorem}[Reachable set for constant inputs]\label{thm:constUncertainInputs}
 The reachable set of the input solution 
 \begin{equation*}
  \hat{\mathcal{R}}_p(\delta) := \int_0^\delta e^{A (\delta - t)} \mathtt{d}t \, \mathcal{U} = \int_0^\delta e^{A t} \mathtt{d}t \, \mathcal{U},
 \end{equation*}
 where $\mathcal{U} = (c_u,g_u\gen{1},\ldots,g_u\gen{q})$ is a zonotope, can be over-approximated by the zonotope
\begin{equation*}
\begin{split}
  \hat{\mathcal{R}}_p(\delta) & \subseteq (\tilde{c}_u, \tilde{g}_u\gen{1},\ldots,\tilde{g}_u\gen{q}) \oplus \hat{\mathcal{R}}_{p,err}, \\
  \tilde{c}_u & = P \, \tilde{V}(c_u) e^{\tilde{H}(c_u) \, \delta} e_1, \\
  \tilde{g}_u\gen{i} & = P \, \tilde{V}(g_u\gen{i}) e^{\tilde{H}(g_u\gen{i}) \, \delta} e_1, \\
  \hat{\mathcal{R}}_{p,err} & = \mathbf{[-1,1]}^n \Big(\tilde{\epsilon}_\textrm{norm}(c_u) + \sum_{i=1}^q \tilde{\epsilon}_\textrm{norm}(g_u\gen{i}) \Big)\delta.
\end{split}
\end{equation*}
\end{theorem}
\begin{proof}
 After inserting the definition of a zonotope (Def.~\ref{def:zonotope}) into $\hat{\mathcal{R}}_p(\delta) := \int_0^\delta e^{A (\delta - t)} \mathtt{d}t \, \mathcal{U}$ and using \eqref{eq:zonoAdditionMultiplication}, we obtain
 \begin{equation*}
  \hat{\mathcal{R}}_p(\delta) = \int_0^\delta e^{A (\delta - t)} \mathtt{d}t \,c_u \oplus \bigoplus_{i=1}^q [-1,1] \int_0^\delta e^{A (\delta - t)} \mathtt{d}t \, g_u\gen{i}.
 \end{equation*}
 Using Lemma~\ref{thm:KrylovErrorInput}, one receives
  \begin{equation*}
  \begin{split}
   \hat{\mathcal{R}}_p(\delta) \subseteq & P\,\tilde{V}(c_u) e^{\tilde{H}(c_u) \, \delta} e_1 \oplus \mathbf{[-1,1]}^n \tilde{\epsilon}_\textrm{norm}(c_u) \, \delta \\
   & \oplus \bigoplus_{i=1}^q \Big([-1,1] P \, \tilde{V}(g_u\gen{i}) e^{\tilde{H}(g_u\gen{i})\, \delta} \, e_1 \\
   & \oplus \mathbf{[-1,1]}^n \tilde{\epsilon}_\textrm{norm}(g_u\gen{i}) \, \delta \Big) \\
   =&  P\, \tilde{V}(c_u) e^{\tilde{H}(c_u) \, \delta} e_1 \\
   & \oplus \bigoplus_{i=1}^q [-1,1] P\, \tilde{V}(g_u\gen{i}) e^{\tilde{H}(g_u\gen{i}) \, \delta} e_1 \\
   & \oplus \underbrace{\mathbf{[-1,1]}^n \Big(\tilde{\epsilon}_\textrm{norm}(c_u) + \sum_{i=1}^q \tilde{\epsilon}_\textrm{norm}(g_u\gen{i}) \Big)\delta}_{= \hat{\mathcal{R}}_{p,err}}.
  \end{split}
 \end{equation*}
 This results in the zonotope of the theorem. Since $\hat{\mathcal{R}}_{p,err}$ is a multidimensional interval, and thus also a zonotope, the addition of the two zonotopes $\hat{\mathcal{R}}_{p,err}$ and $(\tilde{c}_u, \tilde{g}_u\gen{1},\ldots,\tilde{g}_u\gen{q})$ results in the zonotope $\hat{\mathcal{R}}_p(\delta)$.
\end{proof}

The above derivations only hold for constant inputs. We generalize the previous results to arbitrary input trajectories in the following theorem. 
\begin{theorem}[Reachable set for varying inputs] \label{thm:arbitraryInputSet}
 The reachable set due to uncertain inputs
 \begin{equation*}
 \begin{split}
  \mathcal{R}_p(\delta) = \Big\{x_p(\delta) \Big| & x_p(\delta)=\int_0^\delta e^{A(\delta-t)} u(t) \, dt, \\
  & \forall t\in[0, \delta]: u(t) \in \mathcal{U} \Big\}
  \end{split}
 \end{equation*}
can be over-approximated by
\begin{equation*}
  \begin{split}
    \mathcal{R}_p(\delta) \subseteq & P \Big( \bigoplus_{j=1}^\eta \Big\{ \tilde{V}(u)\frac{\tilde{H}^j(u)}{j!} \Big| u\in\mathcal{U}\Big\}\delta^j \\
  \oplus & \Big\{\tilde{\mathcal{E}}(\delta, u) \Big| u\in\mathcal{U}\Big\} \Big) e_1 \\
  \oplus & \mathbf{[-1,1]}^n \Big\{\tilde{\epsilon}_\textrm{norm}(u) \Big| u\in\mathcal{U}\Big\} \delta,
  \end{split}
\end{equation*} 
where $\tilde{\mathcal{E}}(\delta, u)$ is obtained from \eqref{eq:matrixExponentialRemainder} by replacing $H(x_0)$ with $\tilde{H}(u)$.
\end{theorem}
The proof can be found in Appendix~\ref{app:inputSolution}. Obviously, $\hat{\mathcal{R}}_p(\delta) \subseteq \mathcal{R}_p(\delta)$, since $\hat{\mathcal{R}}_p(\delta)$ is exact except for adding $\hat{\mathcal{R}}_{p,err}$, which is identically added to $\mathcal{R}_p(\delta)$, and $\mathcal{R}_p(\delta)$ also contains all solutions for non-constant inputs.

Next, we consider a particular solution of Theorem~\ref{thm:arbitraryInputSet} when $\mathcal{U}$ is a zonotope.

\begin{corollary}[Varying inputs with zonotopic bounds] \label{thm:arbitraryInputSet_zonotope}
 The reachable set due to uncertain inputs
 \begin{equation*}
 \begin{split}
  \mathcal{R}_p(\delta) = \Big\{x_p(\delta) \Big| & x_p(\delta)=\int_0^\delta e^{A(\delta-t)} u(t) \, dt, \\
  & \forall t\in[0, \delta]: u(t) \in \mathcal{U} \Big\}
  \end{split}
 \end{equation*}
 when $\mathcal{U} = (c_u,g_u\gen{1},\ldots,g_u\gen{q})$ is a zonotope, can be over-approximated by
\begin{equation*}
  \begin{split}
    \mathcal{R}_p(\delta) = & (\tilde{c}_u,\tilde{g}_u\gen{1,1},\ldots,\tilde{g}_u\gen{q,1},\tilde{g}_u\gen{1,2},\ldots,\tilde{g}_u\gen{q,2},\ldots,\tilde{g}_u\gen{q,\eta}) \\
    & \oplus \hat{\mathcal{R}}_{p,err}, \\
    \tilde{c}_u & = P \bigg( \sum_{j=1}^\eta  \tilde{V}(c_u)\frac{\tilde{H}^j(c_u)}{j!}\delta^j\bigg) e_1, \\
    \tilde{g}_u\gen{i,j} & = P \tilde{V}(g_u\gen{i})\frac{\tilde{H}^j(g_u\gen{i})}{j!}\delta^j e_1, \\
    \mathcal{R}_{p,err} & = P \bigg(\tilde{\mathcal{E}}(\delta, c_u) \oplus \bigoplus_{i=1}^q \tilde{\mathcal{E}}(\delta, g_u\gen{i}) \bigg) e_1 \\
    & \oplus \mathbf{[-1,1]}^n \Big( \tilde{\epsilon}_\textrm{norm}(c_u) + \sum_{i=1}^q \tilde{\epsilon}_\textrm{norm}(g_u\gen{i}) \Big) \delta.
  \end{split}
\end{equation*}  
\end{corollary}

\begin{proof}
 After inserting the definition of a zonotope (Def.~\ref{def:zonotope}) into the result of Theorem~\ref{thm:arbitraryInputSet}, we obtain
 \begin{equation*}
  \begin{split}
    \mathcal{R}_p(\delta) \subseteq & P \bigg( \bigoplus_{j=1}^\eta \Big( \tilde{V}(c_u)\frac{\tilde{H}^j(c_u)}{j!} \\
    \oplus & \bigoplus_{i=1}^q [-1,1] \tilde{V}(g_u\gen{i})\frac{\tilde{H}^j(g_u\gen{i})}{j!}\Big)\delta^j \\
    \oplus & \underbrace{\tilde{\mathcal{E}}(\delta, c_u) \oplus \bigoplus_{i=1}^q [-1,1]  \tilde{\mathcal{E}}(\delta, g_u\gen{i})}_{= \tilde{\mathcal{E}}(\delta, c_u) \oplus \bigoplus_{i=1}^q \tilde{\mathcal{E}}(\delta, g_u\gen{i})} \bigg) e_1 \\
    \oplus & \underbrace{\mathbf{[-1,1]}^n \Big( \tilde{\epsilon}_\textrm{norm}(c_u) \oplus \bigoplus_{i=1}^q [-1,1] \tilde{\epsilon}_\textrm{norm}(g_u\gen{i}) \Big) \delta}_{=\mathbf{[-1,1]}^n \Big( \tilde{\epsilon}_\textrm{norm}(c_u) + \sum_{i=1}^q \tilde{\epsilon}_\textrm{norm}(g_u\gen{i}) \Big) \delta},
  \end{split}
 \end{equation*} 
 which can be simplified to
 \begin{equation*}
  \begin{split}
    \mathcal{R}_p(\delta) \subseteq & P \bigg( \sum_{j=1}^\eta  \tilde{V}(c_u)\frac{\tilde{H}^j(c_u)}{j!}\delta^j\bigg) e_1 \\
    \oplus & \bigoplus_{j=1}^\eta \bigoplus_{i=1}^q [-1,1] \bigg(P \tilde{V}(g_u\gen{i})\frac{\tilde{H}^j(g_u\gen{i})}{j!}\delta^j\bigg) e_1 \oplus \mathcal{R}_{p,err}.
  \end{split}
 \end{equation*} 
 This results in the zonotope of the theorem. Since $\mathcal{R}_{p,err}$ is a multidimensional interval, and thus also a zonotope, the addition of the two zonotopes $\mathcal{R}_{p,err}$ and $(\tilde{c}_u,\tilde{g}_u\gen{1,1},\ldots,\tilde{g}_u\gen{q,1},\tilde{g}_u\gen{1,2},\ldots,\tilde{g}_u\gen{q,2},\ldots,\tilde{g}_u\gen{q,\eta})$ results in the zonotope $\mathcal{R}_p(\delta)$.
\end{proof}

\section{Propagation} \label{sec:propagation}

So far, we have only described the computation of a reachable set for the first time interval. This section presents how the initial results are propagated for further consecutive time intervals $\tau_i$.

\subsection{Overall Algorithm} \label{sec:overallAlgorithm}

To grasp the implemented propagation scheme more easily, we start with one of the simplest propagation procedures (see \cite[Alg.~1]{Girard2005}): 
 \begin{align}
  \mathcal{R}_{h}(\tau_{k+1}) &= e^{A\delta} \mathcal{R}_{h}(\tau_{k}), \label{eq:homSolPropagation} \\
  \mathcal{R}_{p}(\tau_{k+1}) &= e^{A\delta} \mathcal{R}_{p}(\tau_{k}) \oplus \mathcal{R}_{p}(\tau_{0}). \label{eq:particulateSolGrowing}
 \end{align}
In order to keep the number of generators of zonotopes that are multiplied by the matrix exponential small, the wrapping-free approach from \cite{Girard2006} is used in this work as a basis; we later introduce modifications to make the best use of the proposed computation in the Krylov subspace. The wrapping-free approach modifies the above procedure by introducing the auxiliary reachable set $\mathcal{R}_{b}$, which is enclosed by an axis-aligned box: $\mathcal{R}_{p}(\tau_0) = \mathtt{box}(\mathcal{R}_{b}(\tau_0))$ (see line~\ref{line:firstBoxEnclosure} of Alg.~\ref{alg:reachsetLin} and \eqref{eq:boxEnclosure}). This has the effect that the representation of $\mathcal{R}_{p}(\tau_{k+1})$ in \eqref{eq:particulateSolGrowing} does not grow in complexity due to the Minkowski addition, but stays a simple axis-aligned box.
The modified computation of $\mathcal{R}_{p}(\tau_{k+1})$ according to \cite{Girard2006} is:
 \begin{align} 
  \mathcal{R}_b(\tau_{k+1}) &= e^{A\delta} \mathcal{R}_b(\tau_{k}), \label{eq:boxPropagation} \\
  \mathcal{R}_{p}(\tau_{k+1}) &= \mathcal{R}_p(\tau_{k}) \oplus \mathtt{box}(\mathcal{R}_b(\tau_{k+1})). \label{eq:inputPropagation}
 \end{align}
In order to take advantage of the fact that the Arnoldi iteration does not have to be recomputed for different times as shown in \eqref{eq:expApprox_time}, we further modify the computation of $\mathcal{R}_b(\tau_{k+1})$ to
\begin{equation}
 \mathcal{R}_b(\tau_{k+1}) = e^{A t_k} \mathcal{R}_b(\tau_{0}). \label{eq:boxPropagation_mod} 
\end{equation}
Another modification, which is not proposed in \cite{Girard2006}, is that we change \eqref{eq:homSolPropagation}: instead of computing the homogeneous solution for the first time interval only and then propagating it, we apply Theorem~\ref{thm:timeInterval} in each time step. The reason is that for large systems, one saves much computation time when the matrix exponential multiplication is performed with $\mathcal{R}_{h}(t_{k})$ instead of $\mathcal{R}_{h}(\tau_{k})$ due to a fewer number of generators, while the computation in Theorem~\ref{thm:timeInterval} is negligible when using zonotopes. Since the sets in Alg.~\ref{alg:reachsetLin} are indexed without explicitly indicating time, we distinguish sets representing points in time by an asterisk (see e.g., $\mathcal{R}^{*}_{h,1}$) from the ones representing time intervals.

To sum up, in the Krylov subspace, the computation of \eqref{eq:homSolPropagation} and \eqref{eq:boxPropagation_mod} is carried out by Theorem~\ref{thm:homSolution} (see line~\ref{line:homPropagation} and line~\ref{line:boxPropagation} in Alg.~\ref{alg:reachsetLin}). The propagation of the particulate solution \eqref{eq:inputPropagation} is realized in line~\ref{line:inputPropagation} of Alg.~\ref{alg:reachsetLin} and the homogeneous solution for a time interval is obtained in line~\ref{line:timeIntervalPropagation}. The aggregation of the homogeneous and the particulate solution are performed in line~\ref{line:initAggregation} and line~\ref{line:aggregation} of Alg.~\ref{alg:reachsetLin}. 

When only output values $y(t)=Cx(t) + Du$ are of interest, one can save substantial computation time by first multiplying the orthogonal bases $V(c)$, $V(g\gen{i})$, $\tilde{V}(c_u)$, and $\tilde{V}(g_u\gen{i})$ with $C$ to obtain $V(c):=C V(c)$, $V(g\gen{i}) := C V(g\gen{i})$, $\tilde{V}(c_u) := C \tilde{V}(c_u)$, and $\tilde{V}(g_u\gen{i}) := C\tilde{V}(g_u\gen{i})$ as also presented in e.g., \cite{Bak2018a}. Otherwise, one would first compute the full-dimensional reachable set and afterwards compute the projection instead of directly computing the reachable set of the output.

\begin{algorithm}
\caption{Compute $\mathcal{R}([0,t_f])$} \label{alg:reachsetLin}
\begin{algorithmic}[1]
	\Require State matrix $A$, input matrix $B$, initial set $\mathcal{X}_0$, input set $\mathcal{U}$, time step $\delta$, time horizon $t_f$
	\Ensure $\mathcal{R}([0,t_f])$ 
	\State $\mathcal{R}^*_{h,1} = \mathtt{theorem}\text{\texttt{\_\ref{thm:homSolution}}}(\mathcal{X}_0,A,\delta)$	
	\State $\mathcal{R}^*_{b,1} = \mathtt{corollary}\text{\texttt{\_\ref{thm:arbitraryInputSet_zonotope}}}(\mathcal{U},A,B,\delta)$ \label{line:firstInputSol}
	\State $\mathcal{R}_{p,1} = \mathtt{box}\big(\mathtt{conv}(\mathbf{0}^n,\mathcal{R}^*_{b,1})\big)$ \label{line:firstBoxEnclosure}
	\State $\mathcal{R}_{h,1} = \mathtt{theorem}\text{\texttt{\_\ref{thm:timeInterval}}}(\mathcal{X}_0,\mathcal{R}^*_{h,1},A,\delta)$
	\State $\mathcal{R}_1 = \mathcal{R}_{h,1} \oplus \mathcal{R}_{p,1}$ \label{line:initAggregation}
	
	\For{$k=1 \ldots \lceil t_f/\delta-1 \rceil$}
		\State $\mathcal{R}^*_{h,k+1} = \mathtt{theorem}\text{\texttt{\_\ref{thm:homSolution}}}(\mathcal{X}_0,A,(k+1)\delta)$ \label{line:homPropagation}
		\State $\mathcal{R}^*_{b,k+1} = \mathtt{theorem}\text{\texttt{\_\ref{thm:homSolution}}}(\mathcal{R}^*_{b,1},A,(k+1)\delta)$ \label{line:boxPropagation}
		\State $\mathcal{R}_{p,k+1} = \mathcal{R}_{p,k} \oplus \mathtt{box}(\mathcal{R}^*_{b,k+1})$ \label{line:inputPropagation}
		\State $\mathcal{R}_{h,k+1} = \mathtt{theorem}\text{\texttt{\_\ref{thm:timeInterval}}}(\mathcal{R}^*_{h,k},\mathcal{R}^*_{h,k+1},A,\delta)$ \label{line:timeIntervalPropagation}
		\State $\mathcal{R}_{k+1} = \mathcal{R}_{h,k+1} \oplus \mathcal{R}_{p,k+1}$ \label{line:aggregation}
	\EndFor
	\State $\mathcal{R}([0,t_f])=\bigcup_{k=1}^{t_f/\delta} \mathcal{R}_{k}$
\end{algorithmic}
\end{algorithm}

When only constant inputs are considered, $\mathtt{corollary}\text{\texttt{\_\ref{thm:arbitraryInputSet_zonotope}}}(A,B,\mathcal{U},\delta)$ in line~\ref{line:firstInputSol} is replaced by $\mathtt{theorem}\text{\texttt{\_\ref{thm:constUncertainInputs}}}(A,B,\mathcal{U},\delta)$. Also, when the center $u_c$ of the set of uncertain inputs $\mathcal{U}$ is large compared to the deviation from the center $\mathcal{U}_\Delta := \mathcal{U} \oplus (-u_c)$, one should move the solution of $u_c$ inside the convex hull computation performed in $\mathtt{theorem}\text{\texttt{\_\ref{thm:timeInterval}}}(A,\mathcal{X}_0,\delta)$ in line~\ref{line:timeIntervalPropagation} (see e.g., \cite[Alg.~1]{Althoff2014a}). This, however, is independent from the extension to compute in the Krylov subspace and thus not discussed in this work.

\subsection{Computational Complexity} \label{sec:computationalComplexity}

Let us first consider the computational complexity when applying Alg.~\ref{alg:reachsetLin} without utilizing the Krylov subspace: For each time step, we have to map zonotopes using the matrix exponential, compute the box enclosure as well as the convex hull of zonotopes, and add two zonotopes. We introduce the number of generators of $\mathcal{R}^*_{h,1}$ in Alg.~\ref{alg:reachsetLin} as $p_h$ and the number of generators of $\mathcal{R}^*_{b,1}$ as $p_b$. Thus, the orders of those zonotopes are $o_h = \frac{p_h}{n}$ and $o_b = \frac{p_b}{n}$. The number of required binary operations for each of the previously mentioned high-level operands are listed in Tab.~\ref{tab:zonotopeOperations}. We are interested in the complexity with respect to the dimension $n$, since the number of time steps is typically fixed or given by reaching a fixed-point. As can be seen from Tab.~\ref{tab:zonotopeOperations}, the complexity with respect to $n$ is cubic for linear maps (when $M$ is quadratic as for $e^{A \delta}$), linear for addition, quadratic for over-approximating the convex hull, and quadratic for the box enclosure. Thus, the overall complexity is dominated by the linear map, which has complexity $\mathcal{O}((o_h + o_b) \, n^3)$. Note that according to Alg.~\ref{alg:reachsetLin}, the order of the involved zonotopes does not grow compared to other propagation algorithms (see e.g., \cite{Girard2005, Althoff2010a}).

\begin{table}[b]
	\setlength{\tabcolsep}{4pt}
	\renewcommand{\arraystretch}{1.3}
	\centering
	\caption{Required operations.} 
	\begin{tabular}[c]{ll}
	\hline
		 \textbf{operands} \\ \hline\noalign{\vskip 0.2cm}   
		 \multicolumn{2}{l}{$\mathcal{Z}_h=(c_h, g_h\gen{1},\dots,g_h\gen{p_h}), \hat{\mathcal{Z}}_h=(\hat{c}_h, \hat{g}_h\gen{1},\dots,\hat{g}_h\gen{p_h})$}, \\
 		 \multicolumn{2}{l}{$\mathcal{Z}_p=(c_p, g_p\gen{1},\dots,g_p\gen{p_p}) \subset \mathbb{R}^{n}, M\in\mathbb{R}^{m\times n}$} \\[0.2cm]
	\hline
		 \textbf{operation} & \textbf{nr. of binary operations} \\ \hline\noalign{\vskip 0.2cm}   
		 $M \mathcal{Z}_h$, see \eqref{eq:zonoAdditionMultiplication} & $2 mn(o_h n + 1)$ \\[0.2cm]
		 $\mathcal{Z}_h \oplus \mathcal{Z}_p$, see \eqref{eq:zonoAdditionMultiplication} & $n$ \\[0.2cm]
		 $\mathtt{conv}(\mathcal{Z}_h,\hat{\mathcal{Z}}_h)$, see \eqref{eq:convexHull} & $2(o_h n+1)n$ \\[0.2cm]
		 $\mathtt{box}(\mathcal{Z}_h)$, see \eqref{eq:boxEnclosure} & $o_h n^2$ \\ \hline
	\end{tabular}
	\label{tab:zonotopeOperations}
\end{table}

The only difference when computing in the Krylov subspace is that the complexity of computing $e^{A\delta} \mathcal{Z}$ changes. As discussed above, this is also the operation which determines the overall computational complexity. We are discussing the case when the Krylov order is chosen such that the exponential matrix is computed up to machine precision so that the results have the same accuracy and the comparison is fair. The Arnoldi iteration requires $\xi$ matrix-vector multiplications ($\xi$ is the dimension of the Krylov subspace; see Alg.~\ref{alg:Arnoldi}). However, since the involved matrices are sparse, we can add a sparsity constant $s<<1$ defined as the fraction of non-zero entries so that we obtain $\mathcal{O}(s\xi n^2)$ operations. Further, we require $\xi^2/2$ inner products ($\mathcal{O}(\xi^2 \, n)$) and $\xi^2/2$ other operations with $\mathcal{O}(n)$ resulting in $\mathcal{O}(\xi^2 \, n)$, so that the overall complexity is $\mathcal{O}(\xi^2 \, n) + \mathcal{O}(s\xi n^2)$. 

The Arnoldi iteration has to be computed for $(o_h + o_b)n$ generators, so that the computational complexity amounts to $\mathcal{O}(\xi^2 (o_h + o_b)n^2 + s\xi(o_h + o_b)n^3)$ concerning the Arnoldi iteration. The required computation of each exponential matrix $e^{H\delta}$ in the Krylov subspace only depends on $\xi$ and not on $n$. Finally, we require matrix/matrix computations between $V$ and the exponential matrix $e^{H\delta}$ with complexity $\mathcal{O}(n \xi^2)$. 

Thus, the overall complexity using the Krylov technique is $\mathcal{O}(\xi^2 (o_h + o_b)n^2 + s\xi(o_h + o_b)n^3) + \mathcal{O}(n \xi^2) = \mathcal{O}(\xi^2 (o_h + o_b)n^2 + s\xi(o_h + o_b)n^3)$. As a result, both the classical approach and the Krylov approach are cubic in the number of continuous state variables. However, the Krylov method can be much faster (as demonstrated in the next section) depending on the sparsity of the system matrix $A$, since $\xi << n$ and typically $s\xi << 1$ (compare $\mathcal{O}((o_h + o_b) \, n^3)$ for the classical approach with $\mathcal{O}(\xi^2 (o_h + o_b)n^2 + s\xi(o_h + o_b)n^3)$ for the Krylov approach). Please note that the sparsity typically increases the larger a system is (see e.g., \cite[Fig.~3]{Williams2009}). Also, it should be noted that although $A$ is often very sparse, all entries of $e^{A\delta}$ are typically non-zero. Further, observations show that $\xi$ and $n$ are not linearly related for a constant error $\epsilon$, but that $\frac{n}{\xi}$ grows considerably with increasing $n$ \cite{Wang2017}.

\section{Numerical Experiments} \label{sec:numericalExperiment}

To demonstrate the usefulness of the presented approach, we compute reachable sets of benchmark models in \cite{Tran2017} and perform a formal analysis of a bridge as a representative of a safety-critical structure. To the best knowledge of the author, no work on formally bounding values of a safety-critical structure exists; the most prominent example of a failing safety-critical structure is the collapsed Tacoma Narrows Bridge \cite{Amman1941}.

\subsection{Models} \label{sec:models}

\paragraph{Benchmark Models}

Benchmarks for large linear systems have been proposed in \cite{Tran2017}. We have selected the FOM model and the MNA models; the MNA models are inspired by linear differential algebraic systems $E\dot{x} = Ax(t) + Bu(t)$ in \cite{Chahlaoui2002}. In \cite{Tran2017} only MNA-1 and MNA-5 are presented and the $E$ matrix is chosen as the identity matrix to obtain a ordinary differential equations. We have additionally included the models MNA-2 up to MNA-4 by also choosing $E$ as the identity matrix.

\paragraph{Bridge Model} 

The model of the Roosevelt Lake Bridge (Arizona) is taken from a student work \cite{Pisaroni2012}, which investigated its structural dynamics. A picture\footnote{The picture is taken from \href{https://commons.wikimedia.org/wiki/File:TheodoreRooseveltDam.003.140817.jpg}{commons.wikimedia.org}} of the bridge and the corresponding finite element model are shown in Fig.~\ref{fig:bridge}. The bridge is mostly made out of steel, except for the roadway deck. The dam structure in front of the bridge (see Fig.~\ref{fig:bridgePicture}) creates a special aerodynamic situation, which can excite certain natural frequencies of the bridge \cite{Davenport1990}. The bridge model consists of $445$ nodes connecting bars and beams. The exact model is provided as an example in CORA \cite{Althoff2015a}. Using the proposed reachability analysis, we can investigate the entire range of possible time responses for all kinds of frequencies of the exciting forces at once. Thus, we do not miss a single possible frequency that could cause a problem.

\begin{figure}[h!tb]	
    \centering	
    \footnotesize
      \psfrag{a}[c][c]{\shortstack{reference \\ trajectory}}
      \psfrag{b}[c][c]{\shortstack{vehicle I}}
      
      \subfigure[Picture of the bridge with the dam accelerating winds in the background.] 
	  {\includegraphics[width=\columnwidth]{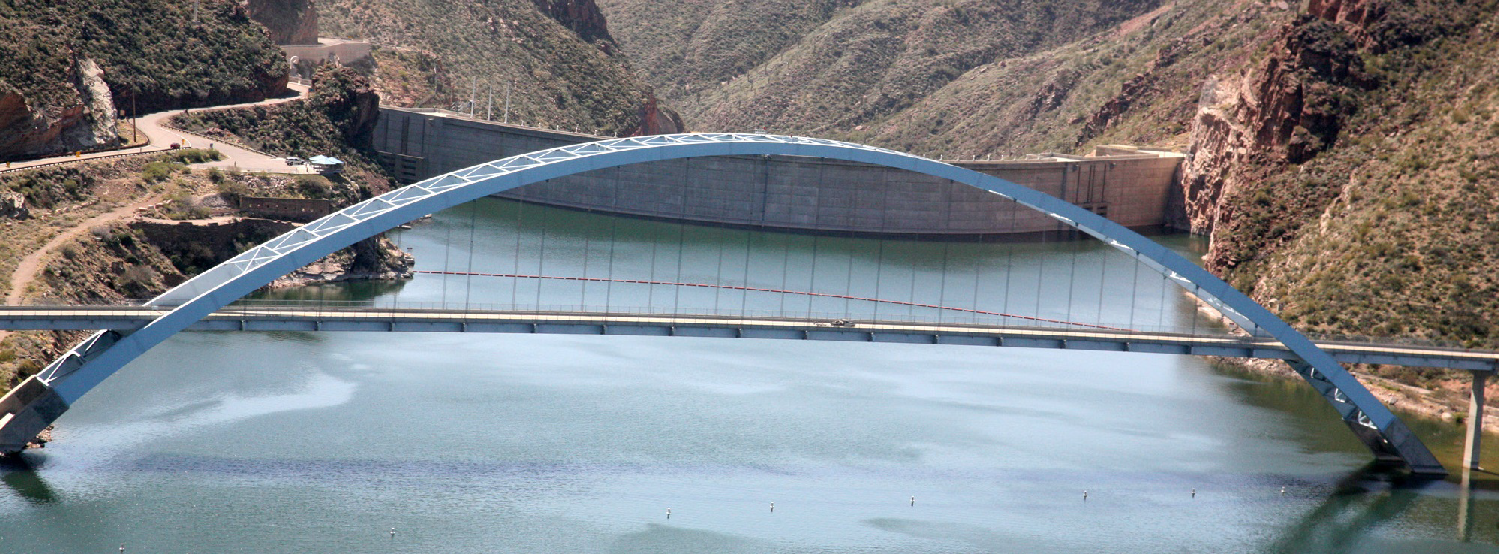} \label{fig:bridgePicture}}
	  
	    \psfrag{d}[c][c]{\shortstack{ego vehicle \\ previously verified part}}
	    \psfrag{h}[c][c]{\shortstack{ego vehicle \\ new intended part}}
	  
      \subfigure[Finite-element model of the bridge (distances are in meters).]        
	  {\includegraphics[width=\columnwidth]{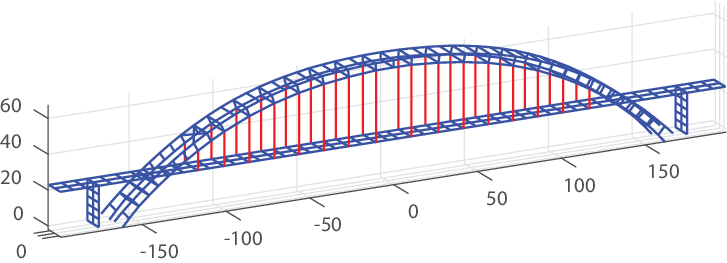} \label{fig:bridgeFEM}}

      \caption{Picture and finite-element model of the Roosevelt Lake Bridge (Arizona).}
			\label{fig:bridge}
\end{figure}

We consider two excitations of the bridge: (1) The bridge is excited by lateral forces simulating the winds acting on the bridge and the street deck and (2) vertical excitation is generated by the street deck caused by moving traffic on the bridge. It should be noted that we consider all kinds of frequencies of these excitations so that no possible solution is missed. Due to the large representation size of the input set $\mathcal{X}_0$ and the input vector $u$ ($5040$ dimensions), we provide those values within the uploaded example in CORA. 

\subsection{Results}

To be as precise as possible, we have chosen the order reduction so that $\|v\| \overline{\epsilon}_\textrm{norm}$ is below the precision of floating point numbers in MATLAB, which is $2.2204\cdot10^{-16}$. The results below are obtained using a standard laptop with an Intel i7-3520M CPU running at 2.90GHz and 8GB of RAM. 

\paragraph{Reachability Analysis}

A summary of the results for reachability analysis is presented in Tab.~\ref{tab:computationTimesAndParameters}. There, the computation times for reachability analysis in the untransformed space and in the Krylov subspace is shown. The computation times include all processes, including determining the proper order of the Krylov subspace. We also show the different computation times depending on whether the full state is required or only the outputs of the system. To better judge the results, we have also added the state dimension, the input dimension, the output dimension, the time step size, and the maximum Krylov order. Since the center and each generator might have different Krylov orders, we just present the maximum value. 

\begin{table*}[h!tb]
  \begin{center}
    \caption{Computation Times and Parameters (OoM: Out of Memory).}
    \label{tab:computationTimesAndParameters}
    \begin{tabular}{lcccccccccccc} \toprule
      & \multicolumn{4}{c}{\textbf{Computation times [s]}} & \multicolumn{8}{c}{\textbf{Model and computation parameters}}  \\ \cmidrule(lr{.75em}){2-5} \cmidrule(lr{.75em}){6-13}
      & \multicolumn{2}{c}{\textbf{States}} & \multicolumn{2}{c}{\textbf{Outputs}} & \multicolumn{4}{c}{\textbf{Dimension}} \\ \cmidrule(lr{.75em}){2-3} \cmidrule(lr{.75em}){4-5} \cmidrule(lr{.75em}){6-9}
      \textbf{Model} & \textbf{Standard} & \textbf{Krylov} & \textbf{Standard} & \textbf{Krylov} & \textbf{State} & \textbf{Input} & \textbf{Output} & \textbf{Krylov} & $\delta$ & $t_f$ & $\zeta$ & $\vartheta$ \\ \midrule
      FOM & 3070 & 308.7 & 3306 & 21.19 & 1006 & 1 & 1 & 120 & 10$^{-4}$ & 10$^{-1}$ & 10 & 0.1 \\ 
      MNA-1 & 83.26 & 94.83 & 85.30 & 82.48 & 578 & 9 & 9 & 240 & 10$^{-5}$ & 10$^{-3}$ & 100 & 0.1 \\ 
      MNA-2 & OoM & OoM & OoM & 1234 & 9223 & 18 & 18 & 700 & 10$^{-5}$ & 10$^{-3}$ & 10 & 0.1 \\ 
      MNA-3 & 4011.99 & 1601 & 3987.14 & 867.1 & 4863 & 22 & 22 & 780 & 10$^{-5}$ & 10$^{-3}$ & 10 & 0.1 \\ 
      MNA-4 & 344.1 & 44.11 & 333.7 & 21.36 & 980 & 4 & 4 & 260 & 10$^{-5}$ & 10$^{-3}$ & 10 & 0.1 \\ 
      MNA-5 & OoM & 2850 & OoM & 180.6 & 10913 & 9 & 9 & 440 & 10$^{-1}$ & 10$^{+1}$ & 10 & 0.1 \\ 
      Bridge & 490912 & 3835 & 490912 & 3835 & 5040 & 2 & 5040 & 380 & 10$^{-7}$ & 10$^{-4}$ & 0.001 & 0.1 \\ 
      \bottomrule
    \end{tabular}
  \end{center}
\end{table*}

In order to determine the proper Krylov order fulfilling the error $\|v\| \overline{\epsilon}_\textrm{norm}$ in this work, we iteratively increment the Krylov order by $20$ starting from $1$. Please note that the step sizes are rather small since we consider the full models and do not reduce them by removing higher frequencies. For all models we have chosen the initial values of the first ten state variables to be uncertain within the interval $\zeta[-1,1]$ and all inputs to be uncertain within the interval $\vartheta[-1,1]$.  

Selected plots of reachable sets over time of the benchmark examples are presented in Fig.~\ref{fig:benchmarks}. The results of the bridge model are illustrated using selected two-dimensional projections of the reachable set in Fig.~\ref{fig:reachSet}.

\begin{figure}[h!tb]	
    \centering	
      \subfigure[FOM model.] 
	  {\includegraphics[width=0.48\columnwidth]{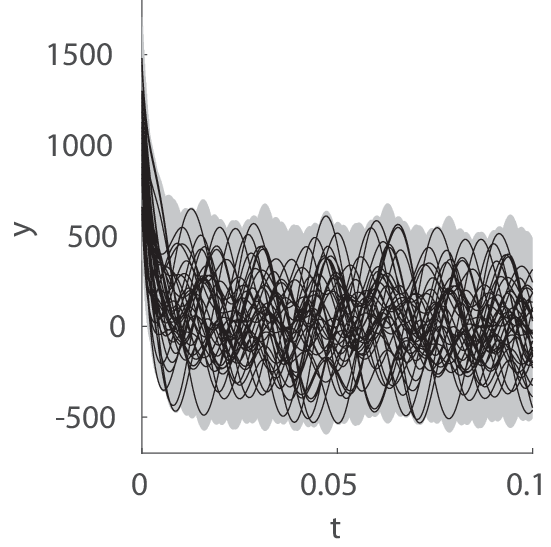} \label{fig:fom}}
      \subfigure[MNA-5 model.]        
	  {\includegraphics[width=0.48\columnwidth]{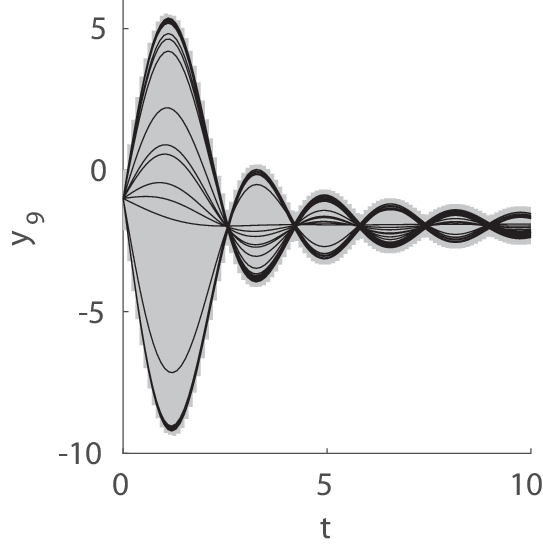} \label{fig:MNA-5}}
      \caption{Reachable set of outputs for selected benchmark models plotted over time. Black lines show random simulations, the gray area shows the reachable set.}
      \label{fig:benchmarks}
\end{figure}

\begin{figure*}[h!tb]	
    \centering	
    \begin{minipage}{2.1\columnwidth}
      \includegraphics[width=0.24\columnwidth]{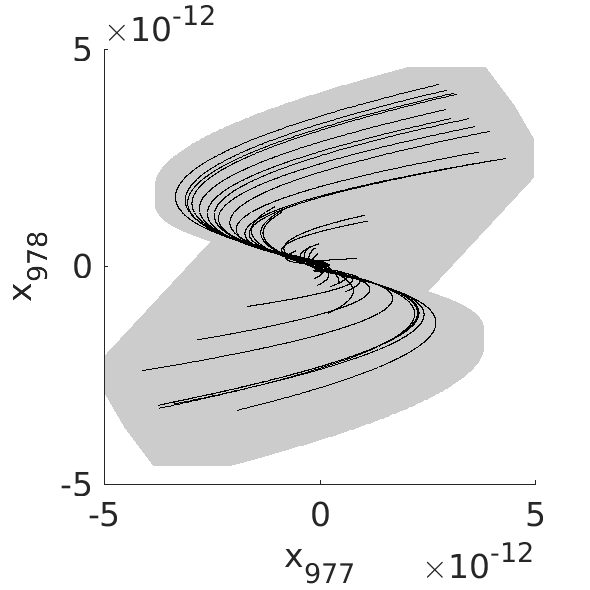} 
      \includegraphics[width=0.24\columnwidth]{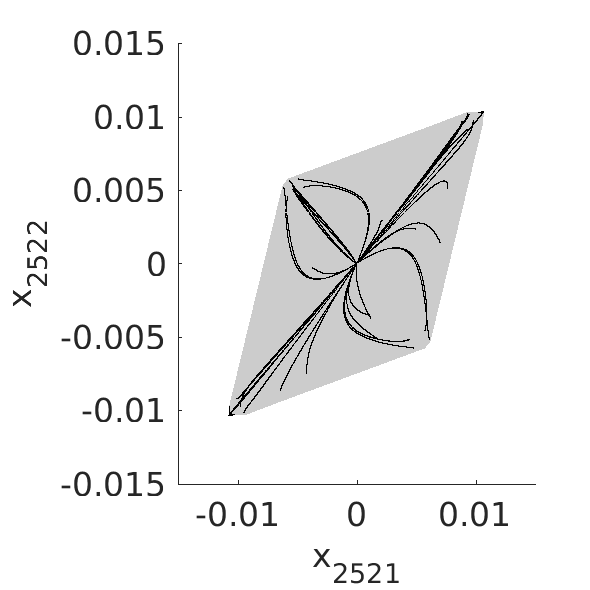} 
      \includegraphics[width=0.24\columnwidth]{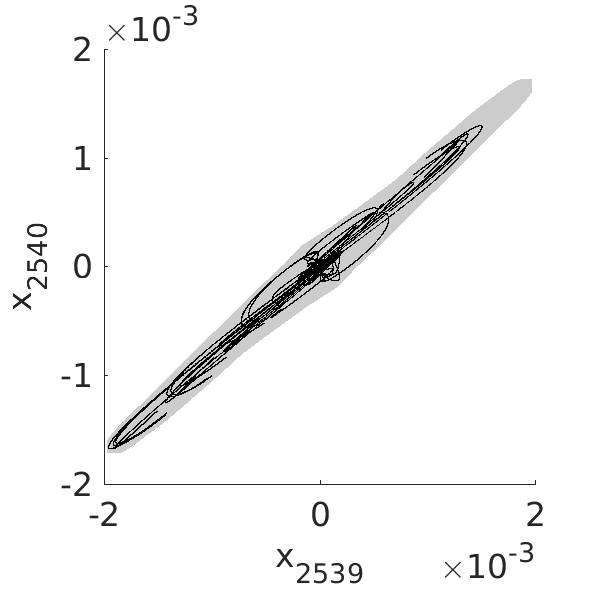} 
      \includegraphics[width=0.24\columnwidth]{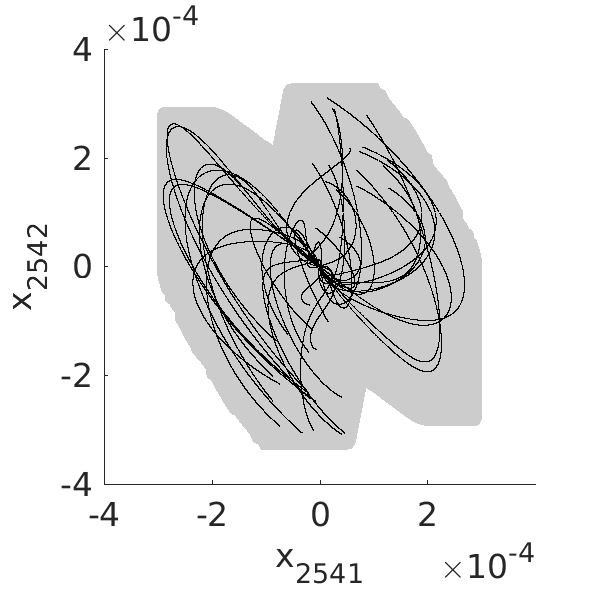} \\[0.2cm] 
      \includegraphics[width=0.24\columnwidth]{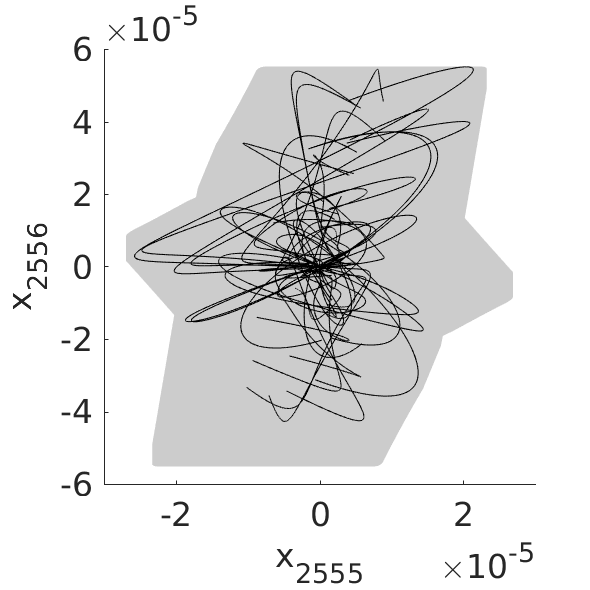} 
      \includegraphics[width=0.24\columnwidth]{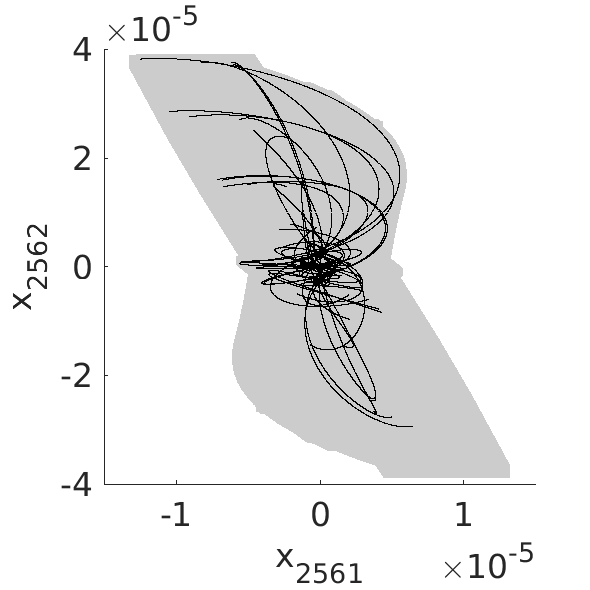} 
      \includegraphics[width=0.24\columnwidth]{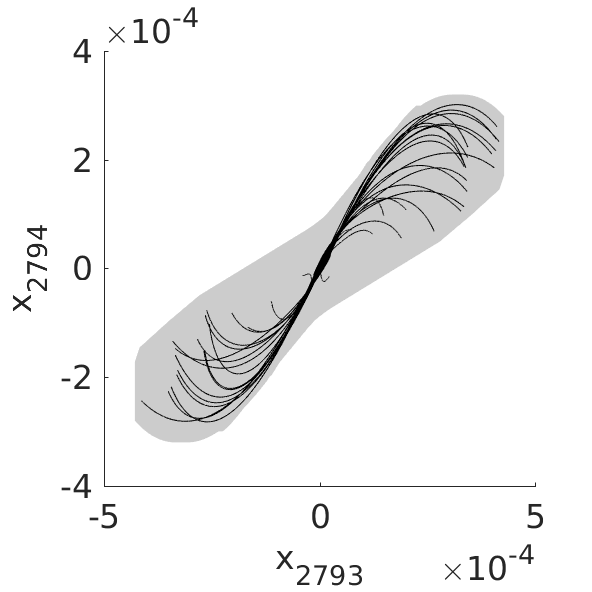} 
      \includegraphics[width=0.24\columnwidth]{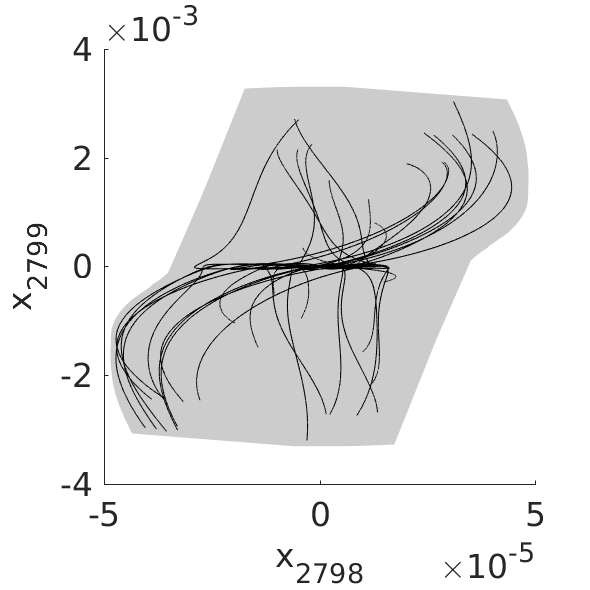} \\[0.2cm] 
      \includegraphics[width=0.24\columnwidth]{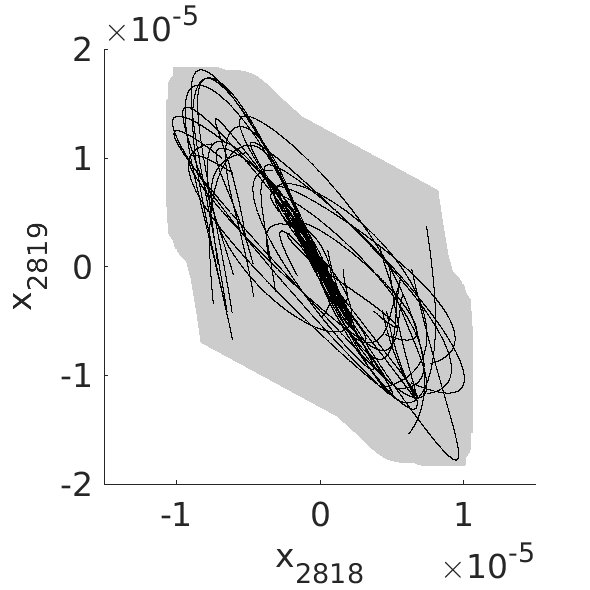} 
      \includegraphics[width=0.24\columnwidth]{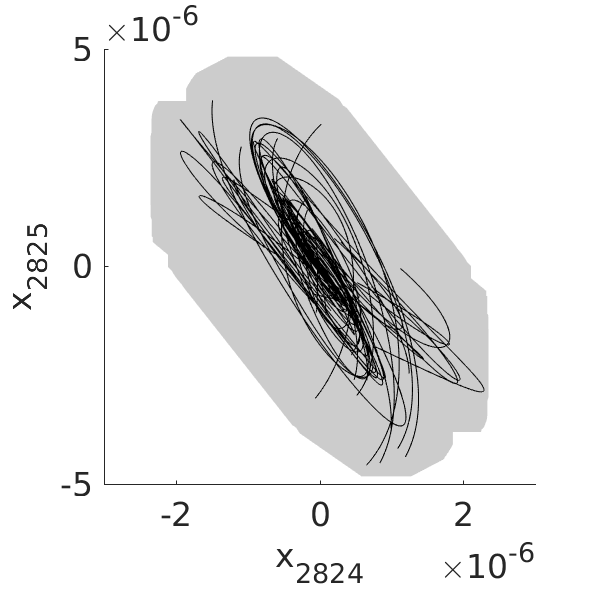} 
      \includegraphics[width=0.24\columnwidth]{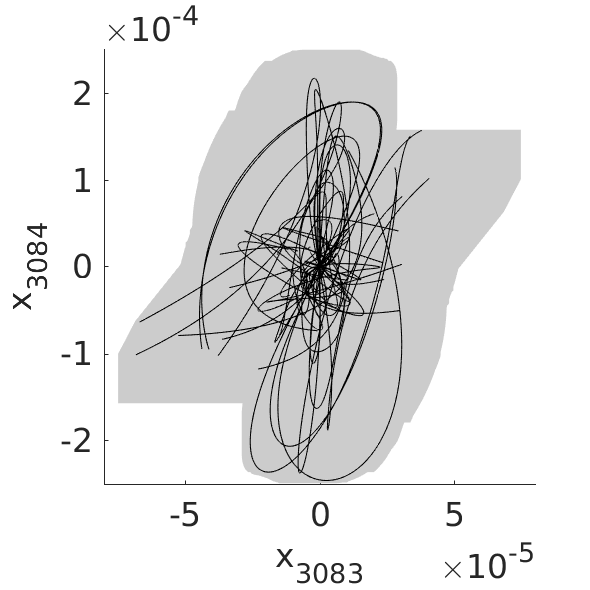} 
      \includegraphics[width=0.24\columnwidth]{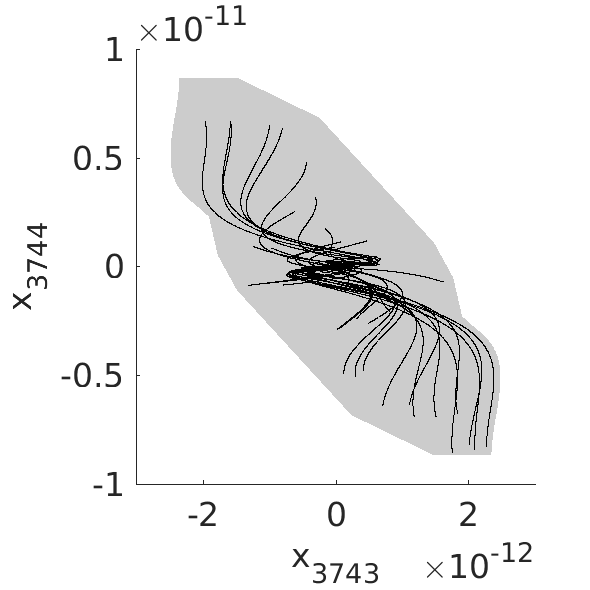} 
    \end{minipage}

      \caption{Two-dimensional projections of reachable sets for selected projections of the bridge model. Black lines show random simulations, the gray area shows the reachable set.}
      \label{fig:reachSet}
\end{figure*}

\paragraph{Verification} While the focus of this work is on reachability analysis, we also briefly demonstrate a possible verification of the bridge for which we check whether all states stay within certain axis-aligned bounds. Since the bounds consist of $2\cdot5040=10080$ values, we will only provide them within CORA. In the example, the reachable set respects the given bounds; checking the bounds for all time interval takes around $20$~s, which is small compared to $3835$~s for the reachability analysis.

\paragraph{Discussion} 

Overall, the computational benefits of the Krylov methods typically improve with the size of the model as shown in Tab.~\ref{tab:computationTimesAndParameters}. When only the output is of importance, the Krylov method is particularly efficient since one can directly include the output matrix $C$ in the Krylov method by using the orthonormal basis $V(v):=C V(v)$ as discussed in Sec.~\ref{sec:overallAlgorithm} (see also \cite{Bak2018a}). This avoids computing the reachable set of all states followed by projecting the result onto the outputs as required by standard methods. For the FOM model and the bridge model, the Krylov method is around 100 times faster, while the improvements for the MNA models are also substantial---only the small MNA-1 model shows similar computation times for both models. 

In general, the computation times of the Kylov method are less predictable compared to the standard approach given the dimension of the system. While the computation times of the standard approach are typically monotonically increasing with the system dimension, the Krylov methods sometimes provides faster results for larger systems.


\section{Conclusions} 
\label{sec:conclusions}

We have presented the first work for computing the over-approximative reachable set of linear systems in the Krylov subspace for arbitrarily-varying, bounded inputs. Unlike most other work on applying reduction techniques for reachability analysis, which compute an output abstraction, we perform a state abstraction and can fully reconstruct the complete reachable sets. When using output abstractions, only the reachable outputs can be over-approximated---if those output-abstraction techniques would be used to reconstruct the whole reachable set, no reduction would be achieved. Due to the strict consideration of error bounds, our approach can be used for formal verification and other formal techniques, such as computation of invariance sets, computation of the region of attraction, optimization of constrained systems with uncertainties, set-based observers, and conformance checking. In our numerical example we have seen a speed up of two orders of magnitude although the approximation was accurate up to floating point precision.


\section*{Acknowledgment}

This work was partly supported by the German Research Foundation (DFG) under grant number AL 1185/5-1.

%
%
%
%

%


\appendices

\section{Proof of Theorem~\ref{thm:arbitraryInputSet}} \label{app:inputSolution}

\begin{proof}
We first compute an over-approximative reachable set resulting from inputs when assuming constant inputs for time intervals $[\hat{t}_{k-1},\hat{t}_{k}[$, where $0=\hat{t}_0 < \hat{t}_1 < \ldots < \hat{t}_{l-1} < \hat{t}_{l}=\delta$, of equal duration $\rho = \hat{t}_{k} - \hat{t}_{k-1}$, so that
\begin{equation*}
 u(t) = \begin{cases}
         u(0) & \text{for } t\in[0,\hat{t}_1[ \\
         u(\hat{t}_1) & \text{for } t\in[\hat{t}_1,\hat{t}_2[ \\
         \vdots & \\
         u(\hat{t}_{l-1}) & \text{for } t\in[\hat{t}_{l-1},\delta].
        \end{cases}
\end{equation*}
In a second step, we let $\rho \to 0$ to obtain an over-approximation for arbitrarily varying inputs. The solution of piecewise constant inputs is obtained from
  \begin{equation*}
  \begin{split}
	  x_\mathtt{p,pw}(\delta) = & \int_{0}^{t_1} e^{A(\delta - t)} dt \, u(0) + \int_{t_1}^{t_2} e^{A(\delta-t)} dt \, u(t_1) \\ 
	  & + \ldots + \int_{t_{l-1}}^{\delta} e^{A(\delta-t)} dt \, u(t_{l-1}).
  \end{split}
  \end{equation*}
When the set of inputs is uncertain within $\mathcal{U}$, we obtain
  \begin{equation}\label{eq:piecewiseConstInputSolution_set}
  \begin{split}
	  \hat{\mathcal{R}}_{p,l}(\delta) = & \int_{0}^{t_1} e^{A(\delta - t)} \, dt \, \mathcal{U} 
	  \oplus \ldots \oplus \int_{t_{l-1}}^{\delta} e^{A(\delta-t)} \, dt \, \mathcal{U}.
  \end{split}
  \end{equation}
  
In order to obtain not only an approximation, but an over-approximation, the solution for a time interval $[t_{k-1},t_{k}[$ is further abstracted. From Corrollary~\ref{thm:KrylovErrorInput_generalTimeBounds} we have that
\begin{equation*}
\begin{split}
  & \hat{\mathcal{R}}_p([t_0, t_e]) := \int_{t_0}^{t_e} e^{A (\delta - t)} \mathtt{d}t \, \mathcal{U} \\
  \subseteq & P \Big\{\tilde{V}(u) \big(e^{\tilde{H}(u) \, (\delta - t_0)} - e^{\tilde{H}(u) \, (\delta - t_e)}\big) \Big| u\in\mathcal{U}\Big\}e_1 \\
  & \oplus \mathbf{[-1,1]}^n \Big\{\tilde{\epsilon}_\textrm{norm}(u) \Big| u\in\mathcal{U}\Big\} (t_e - t_0).
\end{split}
\end{equation*} 
The above over-approximation is rewritten using a finite Taylor series (see \eqref{eq:taylorSeries}):
\begin{equation*}
\begin{split}
  & \hat{\mathcal{R}}_p([t_0, t_e]) \subseteq P \Big\{\tilde{V}(u) \Big(I + \frac{\tilde{H}^1(u)}{1!}(\delta - t_0)^{1} + \ldots \\
  & + \frac{\tilde{H}^\eta(u)}{\eta!}(\delta - t_0)^{\eta} + \tilde{E}(\delta - t_0,u) - \Big(I + \frac{\tilde{H}^1(u)}{1!}(\delta - t_e)^{1}  \\
  & + \ldots + \frac{\tilde{H}^\eta(u)}{\eta!}(\delta - t_e)^{\eta} + \tilde{E}(\delta - t_e,u)\Big)\Big) \Big| u\in\mathcal{U}, \\
  & \tilde{E}(t,u) \in \tilde{\mathcal{E}}(t,u)\Big\}e_1 \oplus \mathbf{[-1,1]}^n \Big\{\tilde{\epsilon}_\textrm{norm}(u) \Big| u\in\mathcal{U}\Big\} (t_e - t_0) \\
  = & P \Big\{\tilde{V}(u) \frac{\tilde{H}^1(u)}{1!}[(\delta - t_0)^{1} - (\delta - t_e)^{1}] + \ldots \\
  & + \tilde{V}(u)\frac{\tilde{H}^\eta(u)}{\eta!}[(\delta - t_0)^{\eta} - (\delta - t_e)^{\eta}] \\
  & + \tilde{E}(\delta - t_0,u) - \tilde{E}(\delta - t_e,u) \Big| u\in\mathcal{U}, \\
  & \tilde{E}(t,u) \in \tilde{\mathcal{E}}(t,u)\Big\}e_1 \oplus \mathbf{[-1,1]}^n \Big\{\tilde{\epsilon}_\textrm{norm}(u) \Big| u\in\mathcal{U}\Big\} (t_e - t_0).
\end{split}
\end{equation*}

Finally, the uncertainty of the input is moved inwards, which results in a further over-approximation:
\begin{equation}\label{eq:partialInputSolutionTaylor}
\begin{split}
  & \hat{\mathcal{R}}_p([t_0, t_e]) \subseteq P \bigg( \\
  & \Big\{\tilde{V}(u) \frac{\tilde{H}^1(u)}{1!} \Big| u\in\mathcal{U}\Big\} [(\delta - t_0)^{1} - (\delta - t_e)^{1}] \oplus \ldots \\
  \oplus & \Big\{ \tilde{V}(u)\frac{\tilde{H}^\eta(u)}{\eta!} \Big| u\in\mathcal{U}\Big\} [(\delta - t_0)^{\eta} - (\delta - t_e)^{\eta}] \\
  \oplus & \Big\{ \tilde{E}(\delta - t_0,u) - \tilde{E}(\delta - t_e,u) \Big| u\in\mathcal{U}, \tilde{E}(t,u) \in \tilde{\mathcal{E}}(t,u)\Big\}\\
  & \bigg) e_1 \oplus \mathbf{[-1,1]}^n \Big\{\tilde{\epsilon}_\textrm{norm}(u) \Big| u\in\mathcal{U}\Big\} (t_e - t_0).
\end{split}
\end{equation}

After introducing 
\begin{equation*}
 \mathcal{D}\gen{j}:= \Big\{ \tilde{V}(u)\frac{\tilde{H}^j(u)}{j!} \Big| u\in\mathcal{U}\Big\},
\end{equation*}
one can rewrite \eqref{eq:partialInputSolutionTaylor} as
\begin{equation}\label{eq:partialInputSolution}
\begin{split}
  & \hat{\mathcal{R}}_p([t_0, t_e]) = \int_{t_0}^{t_e} e^{A (\delta - t)} \mathtt{d}t \, \mathcal{U} \\
  \subseteq & P \bigg( \mathcal{D}\gen{1}[(\delta - t_0)^{1} - (\delta - t_e)^{1}] \oplus \ldots \\
  & \oplus \mathcal{D}\gen{\eta} [(\delta - t_0)^{\eta} - (\delta - t_e)^{\eta}] \\
  & \oplus \Big\{ \tilde{E}(\delta - t_0,u) - \tilde{E}(\delta - t_e,u) \Big| u\in\mathcal{U}, \tilde{E}(t,u) \in \tilde{\mathcal{E}}(t,u)\Big\} \\
  & \bigg) e_1 \oplus \mathbf{[-1,1]}^n \Big\{\tilde{\epsilon}_\textrm{norm}(u) \Big| u\in\mathcal{U}\Big\} (t_e - t_0).
\end{split}	
\end{equation}
 By repeatedly inserting \eqref{eq:partialInputSolution} into \eqref{eq:piecewiseConstInputSolution_set}, rearranging time intervals $(\delta - t_{k}) = t_{l-k}$, and using $\tilde{E}(t,u) \in \tilde{\mathcal{E}}(t, u) = [-\tilde{W}(t, u), \tilde{W}(t, u)]$, where $\tilde{W}(t, u)= \sum_{i=\eta+1}^{\infty} \frac{t^i}{i!}|\tilde{H}(u)|^i$ from \eqref{eq:matrixExponentialRemainder}, we obtain 
 \begin{equation} \label{eq:inputSetPiecewiseConstInput}
  \begin{split}
  \hat{\mathcal{R}}_{p,l}(\delta) & = P \bigg( \\
     & \begin{array}{rllll}
    & (\mathcal{D}\gen{1}(t_l - t_{l-1}) & \oplus \ldots  & \oplus \mathcal{D}\gen{\eta} (t_l^{\eta} - t_{l-1}^{\eta}) \\
  \oplus & (\mathcal{D}\gen{1}(t_{l-1} - t_{l-2}) & \oplus \ldots  & \oplus\mathcal{D}\gen{\eta} (t_{l-1}^{\eta} - t_{l-2}^{\eta})\\
  \oplus & \ldots \\
  \oplus & \underbrace{(\mathcal{D}\gen{1}(t_1 - t_{0})}_{\bigoplus} & \oplus \ldots & 
  \oplus \underbrace{\mathcal{D}\gen{\eta}(t_1^{\eta} - t_{0}^{\eta})}_{\bigoplus} 
  \end{array} \\
  & \begin{array}{rllll}
  \oplus \Big\{ & \bigoplus_{i=\eta+1}^{\infty} \frac{t_l^i - t_{l-1}^i}{i!} \big[-|\tilde{H}(u)|^i, |\tilde{H}(u)|^i \big] \\
  \oplus & \bigoplus_{i=\eta+1}^{\infty} \frac{t_{l-1}^i - t_{l-2}^i}{i!} \big[-|\tilde{H}(u)|^i, |\tilde{H}(u)|^i \big] \\
  \oplus & \ldots \\
  \oplus & \underbrace{\textstyle \bigoplus_{i=\eta+1}^{\infty} \frac{t_{1}^i - t_{0}^i}{i!} \big[-|\tilde{H}(u)|^i, |\tilde{H}(u)|^i \big]}_{\bigoplus} \\
  \end{array} \\
  & \Big| u\in\mathcal{U}\Big\}\bigg) e_1 \\
  & \begin{array}{rllll}
  \oplus & \mathbf{[-1,1]}^n \Big\{\tilde{\epsilon}_\textrm{norm}(u) \Big| u\in\mathcal{U}\Big\} (t_{l} - t_{l-1}) \\
  \oplus & \mathbf{[-1,1]}^n \Big\{\tilde{\epsilon}_\textrm{norm}(u) \Big| u\in\mathcal{U}\Big\} (t_{l-1} - t_{l-2}) \\
  \oplus & \ldots \\
  \oplus & \mathbf{[-1,1]}^n \Big\{\tilde{\epsilon}_\textrm{norm}(u) \Big| u\in\mathcal{U}\Big\} \underbrace{(t_1 - t_{0})}_{\sum}. \\
    \end{array} \\
   \end{split}
  \end{equation}  
The summation symbols indicate that the terms written in one column are summed up. Since the expressions $t_k^{j} - t_{k-1}^{j}$ are positive scalars, the following statement can be used for the summation: 
For any two positive scalars $a,b\in\mathbb{R}^+$ and the convex set $\mathcal{S}$, one can state that $a\mathcal{S}\oplus b\mathcal{S}=(a+b)\mathcal{S}$. From this follows that
\begin{equation}\label{eq:inputSolutionColumnSummation}
\begin{split}
  & \mathcal{D}\gen{j}(t_1^j-t_0^j) \oplus \mathcal{D}\gen{j}(t_2^j-t_1^j)\oplus \ldots \oplus \mathcal{D}\gen{j}(t_l^j-t_{l-1}^j) \\
  &= \mathcal{D}\gen{j}(t_l^j-t_{0}^j) = \mathcal{D}\gen{j}t_l^j = \mathcal{D}\gen{j}\delta^j 
\end{split}
\end{equation}
and similarly for the Taylor remainder terms
\begin{equation} \label{eq:remainderSolutionColumnSummation}
\begin{split}
  & \bigoplus_{i=\eta+1}^{\infty} \frac{t_{1}^i - t_{0}^i}{i!} \big[-|\tilde{H}(u)|^i, |\tilde{H}(u)|^i \big] \oplus \ldots \oplus \\ 
  & \bigoplus_{i=\eta+1}^{\infty} \frac{t_l^i - t_{l-1}^i}{i!} \big[-|\tilde{H}(u)|^i, |\tilde{H}(u)|^i \big] \\
  &= \bigoplus_{i=\eta+1}^{\infty} \frac{t_l^i - t_{0}^i}{i!} \big[-|\tilde{H}(u)|^i, |\tilde{H}(u)|^i \big] = \tilde{\mathcal{E}}(\delta, u)
\end{split}
\end{equation}
and the Krylov error terms
\begin{equation}\label{eq:KrylovErrorSummation}
\begin{split}
  & \mathbf{[-1,1]}^n \Big\{\tilde{\epsilon}_\textrm{norm}(u) \Big| u\in\mathcal{U}\Big\} (t_1 - t_0) \oplus \ldots \oplus  \\
  & \mathbf{[-1,1]}^n \Big\{\tilde{\epsilon}_\textrm{norm}(u) \Big| u\in\mathcal{U}\Big\} (t_l - t_{l-1}) \\
  &= \mathbf{[-1,1]}^n \Big\{\tilde{\epsilon}_\textrm{norm}(u) \Big| u\in\mathcal{U}\Big\} \delta.
\end{split}
\end{equation}
Inserting \eqref{eq:inputSolutionColumnSummation} - \eqref{eq:KrylovErrorSummation} into \eqref{eq:inputSetPiecewiseConstInput} yields
\begin{equation*}
  \begin{split}
    \hat{\mathcal{R}}_{p,l}(\delta) = & P \bigg( \mathcal{D}\gen{1} \delta \oplus \ldots \oplus \mathcal{D}\gen{\eta} \delta^{\eta} \\
  \oplus & \Big\{\tilde{\mathcal{E}}(\delta, u) \Big| u\in\mathcal{U}\Big\} \bigg) e_1 \\
  \oplus & \mathbf{[-1,1]}^n \Big\{\tilde{\epsilon}_\textrm{norm}(u) \Big| u\in\mathcal{U}\Big\} \delta.
   \end{split}
  \end{equation*}  
Since the above result is independent of the number of intermediate time steps $l$, we can choose $l\to\infty$, meaning that 
\begin{equation*}
\begin{split}
  \lim_{l \to \infty} \hat{\mathcal{R}}_{p,l}(\delta) = \mathcal{R}_{p}(\delta).
\end{split}
\end{equation*}
\end{proof}

\section{Error Bound for the Approximation of Matrix Exponentials} \label{app:errorBound}

In this work, we use the a-posteriori bound from \cite[Eq.~4.1]{Jia2015}. A method for efficient computation of this bound is presented in the subsequent proposition.

\begin{proposition}
 The error bound 
 \begin{equation*}
  \forall t \in[0, t_f]: \big\|e^{Ct}v - \|v\| V e^{Ht} e_1\big\| \leq \|v\| \overline{\epsilon}_\textrm{norm} \, t
 \end{equation*}
 in \eqref{eq:KrylovErrorBound} holds for
 \begin{equation*}
 \begin{split}
  \overline{\epsilon}_\textrm{norm} = & h_{\xi+1,\xi} \, \overline{\omega} \, \overline{\varphi}, \\
  \overline{\varphi} = & \begin{cases}
			  \frac{e^{\nu(C)t_f}-1}{\nu(C) t_f} & \text{, for } \nu(C)> 0, \\
			  1 & \text{, otherwise}.
			 \end{cases} \\
  \overline{\omega} = & \max_{k\in\{1,\ldots,\lceil t_f/\delta\rceil \}}(\omega_k), \\
  \omega_k = & e_\xi^T e^{H t_k}e^{H[0,\delta]}e_1, \\
  H[0,\delta] = & \{0.5 H \delta + \beta 0.5 H \delta|\beta \in [-1,1] \},
 \end{split}
 \end{equation*}
 where $H[0,\delta]$ can be represented as a matrix zonotope, i.e., a zonotope whose center and generators are replaced by matrices; the exponential of a matrix zonotope $e^{H[0,\delta]}$ can be computed as presented in \cite{Althoff2011b}. Further, $\nu(C) := \lambda_{max}(0.5(C+C^*))$ returns the largest eigenvalue and $C^*$ is the conjugate transpose of $C$. Please note that the computation of $\nu(C)$ can be done efficiently since $0.5(C+C^*)$ is a real, sparse, symmetric matrix for which efficient eigenvalue computations \cite{Lehoucq1996,Sorensen1992} exist. Results can be obtained up to machine precision \cite{Simon1984} and error bounds are derived in \cite{Paige1980,Paige1976}. Finally, $e_\xi$ is the unit vector whose $\th{\xi}$ value is one and the values of $h_{\xi+1,\xi}$ and $H$ are taken from Alg.~\ref{alg:Arnoldi}.
\end{proposition}
\begin{proof}
 The a-posteriori bound from \cite[Eq.~4.1]{Jia2015} is
\begin{equation} \label{eq:aPosterioriJia}
\begin{split}
 & \big\|e^{Ct_f}v - \|v\| V e^{Ht_f} e_1\big\| \\
 \leq & \|v\| h_{\xi+1,\xi} \max_{t\in[0,t_f]} |e_\xi^T e^{Ht}e_1|\frac{e^{\nu(C)t_f}-1}{\nu(C)}. 
\end{split}
\end{equation}
Please note that some signs are changed compared to \cite[Eq.~4.1]{Jia2015} since in that work, the system is specified as $\dot{x} = e^{-Ct}$ instead of $\dot{x} = e^{Ct}$. 

\paragraph{Error bound is linear in time} 

The error bound can be enclosed by a function that is linear in time. From \eqref{eq:aPosterioriJia} we can derive the following bound: 
\begin{equation} \label{eq:separationOfMaximizations}
\begin{split}
 & \forall t \in[0, t_f]:  \big\|e^{Ct}v  - \|v\| V e^{Ht} e_1\big\| \\
 \leq & \|v\| h_{\xi+1,\xi} \max_{\hat{t}\in[0,t]} \Big(|e_\xi^T e^{H\hat{t}}e_1|\Big)\frac{e^{\nu(C)t}-1}{\nu(C)} \\
 = & t \, \|v\| h_{\xi+1,\xi} \max_{\hat{t}\in[0,t]} \Big(|e_\xi^T e^{H\hat{t}}e_1|\Big)\frac{e^{\nu(C)t}-1}{\nu(C) t} \\
 \leq & t \, \|v\| h_{\xi+1,\xi} \max_{\hat{t}\in[0,t_f]}\Big(|e_\xi^T e^{H\hat{t}}e_1|\Big)\max_{\hat{t}\in[0,t_f]} \bigg(\underbrace{\frac{e^{\nu(C)\hat{t}}-1}{\nu(C) \hat{t}}}_{=:\varphi(\hat{t})}\bigg) \\
\end{split}
\end{equation}
Depending on $\nu(C)$, the function $\varphi(t)$ is monotonically increasing or decreasing. To show this, we insert the Taylor series $e^{\nu(C)t} = \sum_{i=0}^\infty \frac{(\nu(C)t)^n}{n!}$ in the derivative
\begin{equation*}
 \dot{\varphi}(t) = \frac{\nu(C)te^{\nu(C)t}-e^{\nu(C)t}+1}{\nu(C)t^2}
\end{equation*}
so that we obtain
\begin{equation*}
\begin{split}
 \dot{\varphi}(t) =& \frac{\nu(C)t(1+\nu(C)t+\frac{(\nu(C)t)^2}{2!}+\ldots)}{\nu(C)t^2}\\
 -& \frac{(1+\nu(C)t+\frac{(\nu(C)t)^2}{2!}+\ldots)-1}{\nu(C)t^2} \\
 =& \frac{(\nu(C)t)^2+\frac{(\nu(C)t)^3}{2!}+\ldots)-\frac{(\nu(C)t)^2}{2!}-\frac{(\nu(C)t)^3}{3!}-\ldots}{\nu(C)t^2} \\
 =& \sum_{i=0}^\infty \bigg(\frac{1}{(n+1)!} - \frac{1}{(n+2)!}\bigg)\frac{(\nu(C)t)^{n+2}}{\nu(C)t^2} \\
 =& \nu(C) \sum_{i=0}^\infty \bigg(\frac{1}{n+1} - \frac{1}{(n+2)(n+1)}\bigg)\frac{(\nu(C)t)^{n}}{n!} \\
 \in & \nu(C) \sum_{i=0}^\infty ]0,0.5] \frac{(\nu(C)t)^{n}}{n!} \\
 = & \nu(C) ]0,0.5] \sum_{i=0}^\infty  \frac{(\nu(C)t)^{n}}{n!} = \nu(C) \underbrace{]0,0.5] e^{(\nu(C)t)}}_{>0}.
\end{split}
\end{equation*}
Thus, $\varphi(t)$ is monotonically increasing if $\nu(C)>0$, monotonically decreasing for $\nu(C)<0$, and constant for $\nu(C)=0$. As a consequence, the maximum is obtained at $t=t_f$ for $\nu(C)>0$ and at $t=0$ for $\nu(C)\leq0$. For the evaluation of $t=0$, we require the rule of L'H\^{o}pital:
\begin{equation*}
 \lim_{t\to 0} \frac{e^{\nu(C)t}-1}{\nu(C) t} = \lim_{t\to 0} \frac{\nu(C) e^{\nu(C)t}}{\nu(C)} = \lim_{t\to 0} e^{\nu(C)t} = 1.
\end{equation*}
The maximum is then obtained by
\begin{equation} \label{eq:max2}
 \overline{\varphi} = \max_{\hat{t}\in[0,t_f]} \varphi(\hat{t})= \begin{cases}
                                         \frac{e^{\nu(C)t_f}-1}{\nu(C) t_f} & \text{, for } \nu(C)> 0, \\
                                         1 & \text{, otherwise}.
                                        \end{cases}
\end{equation}

\paragraph{Computation of $\max_{\hat{t}\in[0,t_f]} |e_\xi^T e^{H\hat{t}}e_1|$} To compute $\max_{\hat{t}\in[0,t_f]} |e_\xi^T e^{H\hat{t}}e_1|$, we compute
\begin{equation*}
  e^{H\tau_k} = e^{H t_k}e^{H[0,\delta]}
\end{equation*}
as presented in \cite[Eq.~4.10]{Althoff2011b} using the matrix zonotope (see \cite[Eq.~4.8]{Althoff2011b})
\begin{equation*}
 H[0,\delta] = \{0.5 H \delta + \beta 0.5 H \delta|\beta \in [-1,1] \}.
\end{equation*}
Let us also introduce
\begin{equation*}
  \omega_k = e_\xi^T e^{H \tau_k}e_1.
\end{equation*}
Since $\omega_k$ is one-dimensional, the result is the interval $\max_{t\in\tau_k}(e_\xi^T e^{Ht}e_1) = \sup(\omega_k)$. For the entire time horizon, the maximum is simply obtained as 
\begin{equation}\label{eq:max1}
 \overline{\omega} = \max_{t\in[0,t_f]} |e_\xi^T e^{Ht}e_1| = \max_{k\in\{1,\ldots,\lceil t_f/\delta\rceil \}}(\omega_k).
\end{equation}

\paragraph{Final Result} 

using \eqref{eq:max1} and \eqref{eq:max2} in \eqref{eq:separationOfMaximizations} results in the proposition:
\begin{equation*} 
\begin{split}
 \forall t \in[0, t_f]: & \big\|e^{Ct}v - \|v\| V e^{Ht} e_1\big\| \\
 \leq & t \, \|v\| h_{\xi+1,\xi} \, \overline{\omega} \, \overline{\varphi}.
\end{split}
\end{equation*}
\end{proof}
As a short final remark, we would like to mention that we have used the MATLAB function $\mathrm{eigs}$ to obtain $\lambda_{max}$.

\ifCLASSOPTIONcaptionsoff
  \newpage
\fi

\bibliography{althoff_own,althoff_other}
\bibliographystyle{IEEEtran}

\begin{IEEEbiography}[{\includegraphics[width=1in,height=1.25in,clip,keepaspectratio]{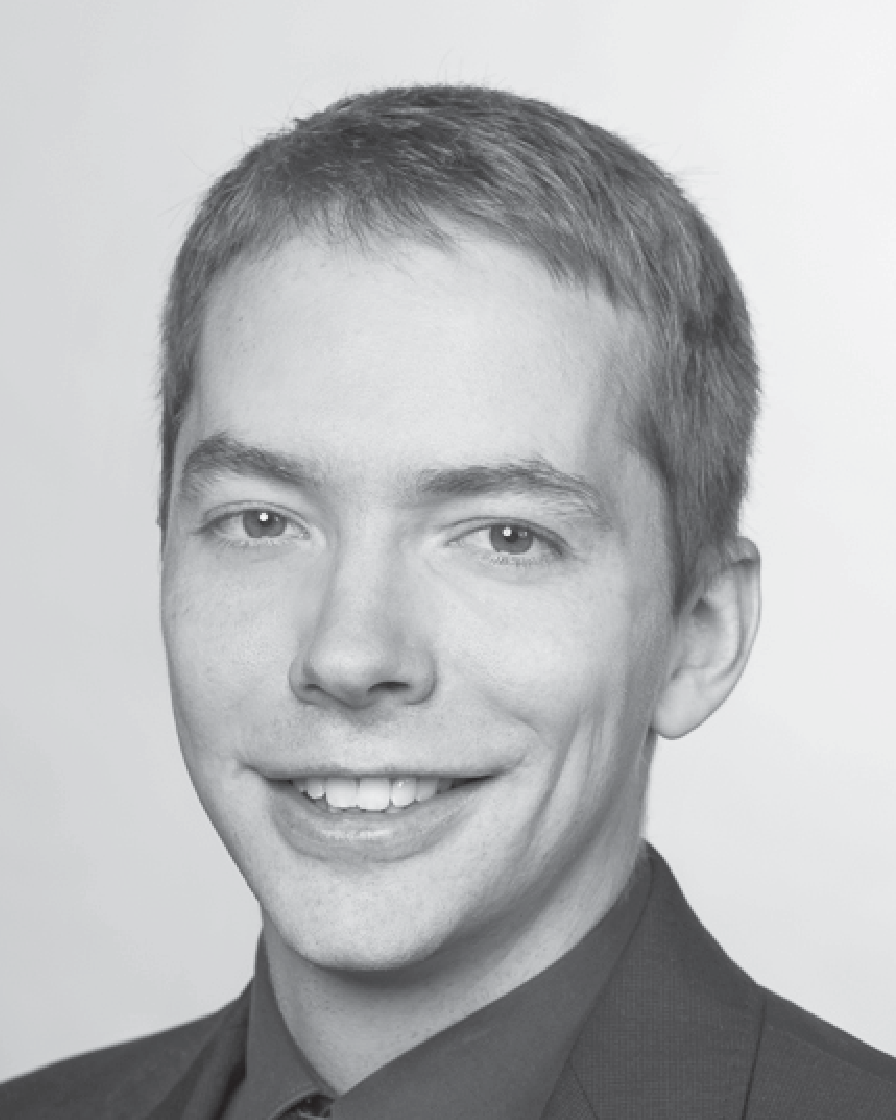}}]
{Matthias Althoff} is an assistant professor in computer science at Technische Universit\"at M\"unchen, Germany. He received his diploma engineering degree in Mechanical
Engineering in 2005, and his Ph.D. degree in Electrical Engineering in
2010, both from Technische Universit\"at M\"unchen, Germany.
From 2010 to 2012 he was a postdoctoral researcher at Carnegie Mellon University,
Pittsburgh, USA, and from 2012 to 2013 an assistant professor at Technische Universit\"at Ilmenau, Germany. His research interests include formal verification of continuous and hybrid systems, reachability analysis, planning algorithms, nonlinear control, automated vehicles, and power systems.
\end{IEEEbiography}


\end{document}